\documentclass[reqno,12pt]{amsart}

\usepackage{amssymb}
\usepackage{latexsym,bm}
\usepackage{hyperref}

\usepackage{amsmath}
\usepackage{amscd}
\usepackage{enumerate}
\usepackage{amsfonts}
\usepackage{epstopdf}
\usepackage{graphicx}
\usepackage[all]{xypic}
\usepackage{mathrsfs}
\usepackage{bm}
\usepackage{setspace}
\usepackage[titletoc]{appendix}
\usepackage{enumitem}
\usepackage{xfrac}

\makeatletter
\def\step{
   \@ifnextchar[ \@step{\@noitemargtrue\@step[\@itemlabel]}}
\def\@step[#1]{\item[#1]\textit{}\hspace*{\dimexpr-\labelwidth-\labelsep}}
\makeatother

\allowdisplaybreaks

\newcommand{\C}{{\mathbb{C}}}
\newcommand{\R}{{\mathbb{R}}}

\renewcommand{\i}{\text{i}}
\newcommand{\e}{\mathrm{e}}
\newcommand{\set}[2]{\left\{#1:~#2\right\}}
\newcommand{\smallo}[1]{\mathrm{o}\left({#1}\right)}
\newcommand{\bigo}[1]{\mathrm{O}\left({#1}\right)}
\newcommand{\abs}[1]{\left|{#1}\right|}
\newcommand{\sts}[1]{\left({#1}\right)}
\newcommand{\ltl}[1]{\left\{{#1}\right\}}
\newcommand{\ntn}[1]{\left[{#1}\right]}
\newcommand{\dual}[2]{\left\langle{#1},\,{#2}\right\rangle}
\newcommand{\act}[2]{\left\langle{#1},\,{#2}\right\rangle}
\newcommand{\biact}[3]{\left\langle{#1}\,{#2},\,{#3}\right\rangle}
\newcommand{\bilin}[2]{\left({#1}, {#2}\right)}
\newcommand{\inprod}[2]{\left({#1},\,{#2}\right)}
\newcommand{\U}[4]{ {\mathcal{U}\left({#1}, \,{#2}, \,{#3}, \,{#4}\right)} }
\newcommand{\normhone}[1]{\left\|#1\right\|_{H^1(\R)}}

\newcommand{\norm}[1]{\|#1\|}
\renewcommand{\d}{\;\mathrm{d}}
\newcommand{\dx}{\;\mathrm{d}x}
\newcommand{\ess}{\mathrm{ess}}
\newcommand{\spn}{\mathrm{span}}

\newcommand{\ift}{\mathrm{IFT}}
\newcommand{\q}{\mathbf{q}}
\newcommand{\supp}[1]{\mathrm{supp}{#1}}
\renewcommand{\S}{\mathcal{S}}
\newcommand{\J}{\mathcal{J}}
\newcommand{\p}{\mathtt p}
\renewcommand{\geq}{\geqslant}
\renewcommand{\leq}{\leqslant}

\makeatletter

\newcommand{\Rmnum}[1]{\expandafter\@slowromancap\romannumeral #1@}

\makeatother
\newcommand{\eps}{{\varepsilon}}

\theoremstyle{plain}
\newtheorem{theorem}{Theorem}
\newtheorem{proposition}[theorem]{Proposition}
\newtheorem{lemma}[theorem]{Lemma}
\newtheorem{corollary}[theorem]{Corollary}

\theoremstyle{definition}

\newtheorem{remark}[theorem]{Remark}

\usepackage{setspace}

\numberwithin{equation}{section}
\numberwithin{theorem}{section}

\setenumerate{fullwidth,itemindent=\parindent,
    listparindent=\parindent,itemsep=0ex,
    partopsep=0pt,parsep=0ex}

\begin{document}

\onehalfspacing

\title[gDNLS]
{Stability of the sum of two solitary waves for (gDNLS) in the energy space}

\author[]{Xingdong Tang}
\address{\hskip-1.15em Xingdong Tang \hfill\newline Beijing Computational Science Research Center, \hfill\newline No. 10 West Dongbeiwang Road, Haidian
District, Beijing, China, 100193,}
\email{txd@csrc.ac.cn}

\author[]{Guixiang Xu}
\address{\hskip-1.15em Guixiang Xu \hfill\newline Institute of
Applied Physics and Computational Mathematics, \hfill\newline P. O.
Box 8009,\ Beijing,\ China,\ 100088.}
\email{xu\_guixiang@iapcm.ac.cn}



\begin{abstract}
In this paper, we continue the study in \cite{MiaoTX:DNLS:Stab}. We use the perturbation argument, modulational analysis and the energy argument in \cite{MartelMT:Stab:gKdV, MartelMT:Stab:NLS} to show the stability of the sum of two
solitary waves with weak interactions for the generalized derivative Schr\"{o}dinger equation (gDNLS) in the energy space. Here (gDNLS) hasn't the Galilean transformation invariance, the pseudo-conformal invariance and the gauge transformation invariance, and the case $\sigma>1$ we considered corresponds to the $L^2$-supercritical case.
\end{abstract}

\maketitle


\section{Introduction}
In this paper, we consider the stability of the solitary waves for the generalized derivative Schr{\"o}dinger equation (gDNLS for short) in $H^1(\R)$
\begin{equation}\label{gDNLS}
\left\{
\begin{aligned}
& \i u_t+u_{xx}+\i |u|^{2\sigma}u_x=0,&\sts{t,x}\in\R\times\R,
\\
& u\sts{0,x}=u_0\sts{x}\in H^1\sts{\R},&
\end{aligned}
\right.
\end{equation}
where $u$ is a complex valued function of $\sts{t,x}\in\R\times\R$, $\sigma>0$.
With $\sigma=1$, \eqref{gDNLS} has appeared as a model for Alfv\'{e}n waves in plasma physics \cite{M-PHY,PassS-PHY,SuSu-book}.

The equation \eqref{gDNLS} is $\dot H^{\frac{\sigma-1}{2\sigma}}$ critical since the scaling
transformation
\begin{align*}u(t,x)\mapsto u_{\lambda}(t,x)=\lambda^{\frac{1}{2\sigma}}u(\lambda^2t, \lambda x)
\end{align*}
leaves both \eqref{gDNLS} and $\dot H^{\frac{\sigma-1}{2\sigma}}$-norm invariant. The mass, momentum
and energy of the solution $u(t,x)$ of \eqref{gDNLS} are defined as following
\begin{align*}
M(u)(t)= & \frac12 \int |u(t,x)|^2 \dx, \\
P(u)(t)=& -\frac{1}{2}\Im\int\sts{\bar{u}u_x}\sts{t,x}\dx ,\\
E(u)(t)=& \frac{1}{2}\int|u_x\sts{t,x}|^2\; \dx
  +\frac{1}{2(\sigma+1)}
  \Im\int\sts{|u|^{2\sigma}\bar{u}u_x}\sts{t,x} \dx.
\end{align*}
They are conserved under the flow \eqref{gDNLS} according to the
phase rotation invariance, spatial translation invariance and time
translation invariance respectively. Compared with
nonlinear Schr\"{o}dinger equation, the equation \eqref{gDNLS} doesn't enjoy the Galilean
invariance and pseudo-conformal invariance any more.

Local well-posedness result for \eqref{gDNLS} with $\sigma\geq 1$ in $H^1(\R)$ has
been worked out by Hayashi and Ozawa \cite{HaOz:gDNLS}. They
combined the compactness method with $L^4_IW^{1,\infty}(\R)$
estimate to construct the local-in-time solution with arbitrary initial data
in the energy space. Since \eqref{gDNLS}
is $\dot H^1$-subcritical, the maximal lifespan interval only depends on
the $H^1$ norm of initial data. More precisely, we have

\begin{theorem}[\cite{HaOz:gDNLS}]
\label{Thm:LWP} Let $\sigma\geq 1$. For any $u_0 \in H^1(\R)$ and $t_0 \in \R$, there
exists a unique maximal-lifespan solution $u:I\times \R \rightarrow
\C$ to \eqref{gDNLS} with $u(t_0)=u_0$, the map $u_0\rightarrow u$ is continuous from $H^1(\R)$ to $C(I, H^1(\R))\cap L^4_{loc}(I, W^{1,\infty}(\R))$.  Moreover, the solution also
has the following properties:
\begin{enumerate}
\item $I$ is an open neighborhood of $t_0$.

\item The mass, momentum and energy are conserved, that is, for all
$t\in I$,
\begin{align*}
M(u)(t)=M(u)(t_0), \;\; P(u)(t)=P(u)(t_0),\;\; E(u)(t)=E(u)(t_0).
\end{align*}

\item If\; $\sup(I)<+\infty$,  $(\text{or}\, \inf(I)>-\infty)$ , then
\begin{align*}
\lim_{t\rightarrow\sup(I)}\big\|\partial_x
u(t)\big\|_{L^2}=+\infty,\;\; \left(
\lim_{t\rightarrow\inf(I)}\big\|\partial_x u(t)\big\|_{L^2}=+\infty,
respectively.\right)
\end{align*}

\item If $\big\| u_0 \big\|_{H^1}$ is sufficiently small,
then $u$ is a global solution.
\end{enumerate}
\end{theorem}

The local well-posedness result of \eqref{gDNLS} with $\sigma \geq 1$ in $H^{1/2}\sts{\R}$ is due to Takaoka \cite{Takaoka:DNLS:LWP} by Fourier restriction norm method and gauge transformation and
Santos \cite{San:gDNLS:LWP} by local smoothing effect of the Schr\"{o}dinger operator. The different features between the case $\sigma=1$ and the case $\sigma>1$ are that the former is the integrable system and has the gauge transformation invariance.  In addtion, there are some numerical stability analysis and blowup results of \eqref{gDNLS} in the energy space, please refer to \cite{CherSS:gDNLS:localStruct,LiuSS:gDNLS:Stab}.

At the same time, it is well-known in \cite{KaupN:DNLS:soliton,LiuSS:gDNLS:Stab} that the equation \eqref{gDNLS} has a two-parameter family of solitary wave solutions of the form
\begin{align*}
  u\sts{t,x}=Q_{\omega,c}\sts{x-ct}\e^{\i\omega t},
\end{align*}
where $4\omega > c^2$,
\begin{align}
\label{Q}
  Q_{\omega,c}\sts{x}=\Phi_{\omega,c}(x)\exp \left\{
    \i \frac{c}{2}x-\frac{\i}{2\sigma +
      2}\int^{x}_{-\infty}\Phi_{\omega,c}^{2\sigma}(y)\d y\right\},
\end{align}
and
\begin{equation}
  \label{phi}
  \Phi_{\omega,c}(x)
  =
  \sts{
    \frac{(\sigma+1)(4\omega-c^2)}{
        2\sqrt{\omega}(\cosh(\sigma\sqrt{4\omega-c^2}x)-\frac{c}{2\sqrt{\omega}})
        }
    }^{\frac{1}{2\sigma}}
\end{equation}
is the unique positive solution of
\begin{equation*}
  - \partial_{x}^2\Phi_{\omega,c}  + (\omega -
  \frac{c^2}{4})\Phi_{\omega,c} +\frac{c}{2}|\Phi_{\omega,c}|^{2\sigma}\Phi_{\omega,c} - \frac{2\sigma + 1}{(2\sigma +
    2)^2}|\Phi_{\omega,c}|^{4\sigma}\Phi_{\omega,c}  = 0,
\end{equation*}
up to phase rotation and spatial translation invariance. By the stability criteria in \cite{GrillSS:Stable:87} \cite{GSS:NLS:Stab2}, it was shown that they are orbitally stable when $\sigma\in (0, 1)$, and orbitally unstable when $\sigma\geq 2$ in \cite{LiuSS:gDNLS:Stab}.

For the case $\sigma=1$ and $4\omega > c^2$. On one hand, by  the convex analysis in \cite{PaySatt:Instab}, the structure analysis\footnote{Since the nonlinearity has derivative in \eqref{gDNLS}, the structure analysis of the solitary waves before using the variational argument in \cite{AmbMal:book, IbrMN:NLKG:Scat, NakanishiSchlag:Book:invariant manifold} is used to transform the quasilinear problem into the semilinear problem in principle in \cite{MiaoTX:DNLS:Exist}.} and the variational characterization of the solitary waves, Miao, Tang and Xu obtained the global wellposedness result in some invariant subset $K^{+}$ of the energy space in \cite{MiaoTX:DNLS:Exist}, where the construction of $K^{+}$ is related to the variational characterization of the solitary wave. On the other hand,
 Colin and Ohta \cite{ColinOhta-DNLS} made use of the concentration compactness argument and  proved that the above solitary waves are orbitally stable in the energy space. Because \eqref{gDNLS} is an integrable system,  Nakamura and Chen obtained the explicit
formula of the multi-soliton solutions of \eqref{gDNLS} in \cite{NakChen:DNLS:Mulsol} by Hirota's
bilinear transform method. Recently, Miao, Tang and Xu in \cite{MiaoTX:DNLS:Stab} and Le Coz and Wu in \cite{LeWu:DNLS} independently showed the stability of the sum of the multi-soliton waves with weak interactions in the energy space, where the arguments are both based on the perturbation argument, the modulation stability and the energy argument in \cite{MartelMT:Stab:gKdV, MartelMT:Stab:NLS}.

For the case $\sigma\in(1,2)$ and $4\omega > c^2$.
Fukaya, Hayashi and Inui  showed the variational characterization of the solitary waves of \eqref{gDNLS} in \cite{FuHaIn:gDNLS:GWP}, i.e.
$Q_{\omega,c}$ is a minimizer of the following problem:
\begin{align}
\label{d}
    d\sts{\omega,c}
    =
    \inf\set{S_{\omega,c}\sts{\varphi} }{ \varphi\in H^1\setminus\ltl{0},
    ~~ K_{\omega,c}\sts{\varphi}=0  }
\end{align}
where the action functional is defined by
\begin{equation}
  \label{S}
  S_{\omega,c}(\varphi)=E(\varphi)+ \omega M(\varphi) + c P(\varphi),
\end{equation}
and the scaling derivative functional is defined by
\begin{equation}
  \label{K}
  K_{\omega,c}(\varphi)
  =
  \left.\frac{\d}{\d \lambda}S_{\omega,c}(\lambda\varphi)\right|_{\lambda=1}.
\end{equation}
In addition, they also obtained  the global well-posedness
of the solution to \eqref{gDNLS} in the similar invariant subset $K^{+}$ as that in \cite{MiaoTX:DNLS:Exist}.

Next we consider its stability in the energy space. Let $z_0=z_0(\sigma)$ be the unique solution in \(\sts{-1,1}\) of $ F(z_0;\sigma) =0,$
    where $F(z;\sigma)$ is defined by
    \begin{align}
    \nonumber
        F(z; \sigma) = & (\sigma-1)^2 \ltl{\int_0^\infty (\cosh y -z)^{-\frac{1}{\sigma}} \d y}^2\\
        &- \ltl{\int_0^\infty (\cosh y -z)^{-\frac{1}{\sigma}-1}(z
          \cosh y -1)\d y }^2.\label{eq:sig-z}
    \end{align}
By the stability criteria in \cite{GSS:NLS:Stab2}, Liu, Simpson and Sulem numerically showed that the solitary wave is stable for $c\in (-2\sqrt{\omega}, 2z_0\sqrt{\omega})$   and unstable for $c\in (2z_0\sqrt{\omega}, 2\sqrt{\omega})$ in \cite{LiuSS:gDNLS:Stab}. That is,
\begin{theorem}[\cite{LiuSS:gDNLS:Stab}]
\label{mainthm:single}Let $\sigma\in\sts{1,2}$ and $z_0=z_0(\sigma)\in\sts{-1,1}$ satisfy $ F(z_0;\sigma) =0,$
    where $F(z;\sigma)$ is defined by \eqref{eq:sig-z}. Let $(\omega^0, c^0)\in \R^2,$ satisfy $c^0\in\sts{-2\sqrt{\omega^0} ,2 z_0 \sqrt{\omega^0}}$, the solitary wave  $
Q_{\omega^0,c^0}(x-c^0t)e^{i\omega^0 t } $ to \eqref{gDNLS} is
orbitally stable in the energy space. That is, for any $\epsilon>0$,
there exists $\delta>0$ such that if $u_0\in H^1(\R)$ satisfies
$$\big\|u_0(\cdot)-
Q_{\omega^0,c^0}(\cdot-x^0)e^{i\gamma^0
}\big\|_{H^1(\R)}<\delta$$ for some $(x^0, \gamma^0)\in \R^2$, then
the solution $u(t)$ of \eqref{gDNLS} exists globally in time and
satisfies
\begin{align*}
\sup_{t\geq0}\inf_{(y, \gamma)\in \R^2}\big\|u(t,\cdot)-
Q_{\omega^0,c^0}(\cdot-y)e^{i\gamma
}\big\|_{H^1(\R)}<\eps.
\end{align*}
\end{theorem}

\begin{remark}
For the case $\sigma\in(3/2, 2)$ and $ c^0=2z_0\sqrt{\omega^0}$, Fukaya shown that the traveling wave is still unstable in \cite{Fukaya:gDNLS:bdline}. It is notice that it are still open problem whether the solitary waves with any $\sigma>0$ and the critical case $ c^0= 2\sqrt{\omega^0}$ are stable or not. In fact, the solitary waves with the critical parameter $c^0= 2\sqrt{\omega^0}$ have polynomial decay, and the difficulty is that there is no the spectral gap about the linearized operator around the solitary wave.
\end{remark}

In this paper, we consider the stability of the sum of two solitary waves for \eqref{gDNLS} with $\sigma\in (1,2)$ and $c^0_k\in (-2\sqrt{\omega^0_k}, 2z_0(\sigma)\sqrt{\omega^0_k})$, $k=1, 2$. As far as we know, the integrability (non-integrability) of \eqref{gDNLS} is not clear,
the existence (nonexistence) of the explicit multi-solition solutions is not
obvious. Here we use the argument in \cite{MiaoTX:DNLS:Stab} (see also \cite{LeWu:DNLS, MartelMT:Stab:gKdV, MartelMT:Stab:NLS}) and the references therein.

The main result is the following result.

\begin{theorem}
\label{mainthm}Let $\sigma\in\sts{1,2}$ and $z_0=z_0(\sigma)\in\sts{0,1}$ satisfy $ F(z_0;\sigma) =0,$
    where $F(z;\sigma)$ is defined by \eqref{eq:sig-z}. Let $(\omega^0_k, c^0_k)\in \R^2,$ $k=1,2$ satisfy
    \begin{enumerate}
    \item[$(a)$] Nonlinear stability:
        $c^0_k\in\sts{-2\sqrt{\omega^0_k},2 z_0 \sqrt{\omega^0_k}}$ for $k=1, 2$.

    \item[$(b)$] Technical assumption: $\frac{\omega^0_2-\omega^0_1}{c^0_2-c^0_1}>0.$

    \item[$(c)$] Relative speed: $c^0_1 <
    \frac{\omega^0_2-\omega^0_1}{c^0_2-c^0_1}$, and $
    4\frac{\omega^0_2-\omega^0_1}{c^0_2-c^0_1}<c^0_2$.
    \end{enumerate}
    Then there exist positive numbers $C$, $\delta_0$, $\theta_0$ and $L_0$, such that if
    $0<\delta<\delta_0,\; L>L_0$ and
    \begin{align*}
    \left\|u_0(\cdot)-\sum^2_{k=1}Q_{\omega^0_k,
    c^0_k}(\cdot-x^0_k)e^{i\gamma^0_k}\right\|_{H^1(\R)} \leq \delta,
    \end{align*}
    with $ x^0_2-x^0_1>L$, then the solution $u(t)$ of \eqref{gDNLS}
    exists globally in time and there exist functions
    $x_k(t)$ and $\gamma_k(t)$, $k=1,2$ such that for any $t\geq 0$,
    \begin{align*}
    \left\|u(t,\cdot)-\sum^2_{k=1}Q_{\omega^0_k,
    c^0_k}(\cdot-x_k(t))e^{i\gamma_k(t)}\right\|_{H^1(\R)} \leq
    C\left(\delta+e^{-\theta_0 \frac{L}{2}}\right).
    \end{align*}
\end{theorem}

\begin{remark}
\begin{enumerate}
  \item \textit{ The function $F(z;\sigma)$ and existence of $z_0.$}
  In order to use the abstract functional analysis argument
  in \cite{GrillSS:Stable:87,GSS:NLS:Stab2}, Liu, Simpson and Sulem introduced  the function $F(z;\sigma)$  to obtain the stability
   (instability) of single soliton solutions of \eqref{gDNLS}
   in \cite{LiuSS:gDNLS:Stab}. The function $F(z;\sigma)$ is closely related to the determinant of the Hessian $d''\sts{\omega,c}$. It numerically turns out that for any fixed $\sigma\in\sts{1,2}$, the function $F(z;\sigma)$ is monotonically decreasing with respect to $z$ and has exactly one root $z_0$ in the interval $\sts{-1,1}.$ See Figure \ref{fig:Fcurves}.
   \begin{figure}[ht]
\includegraphics[width=8cm]{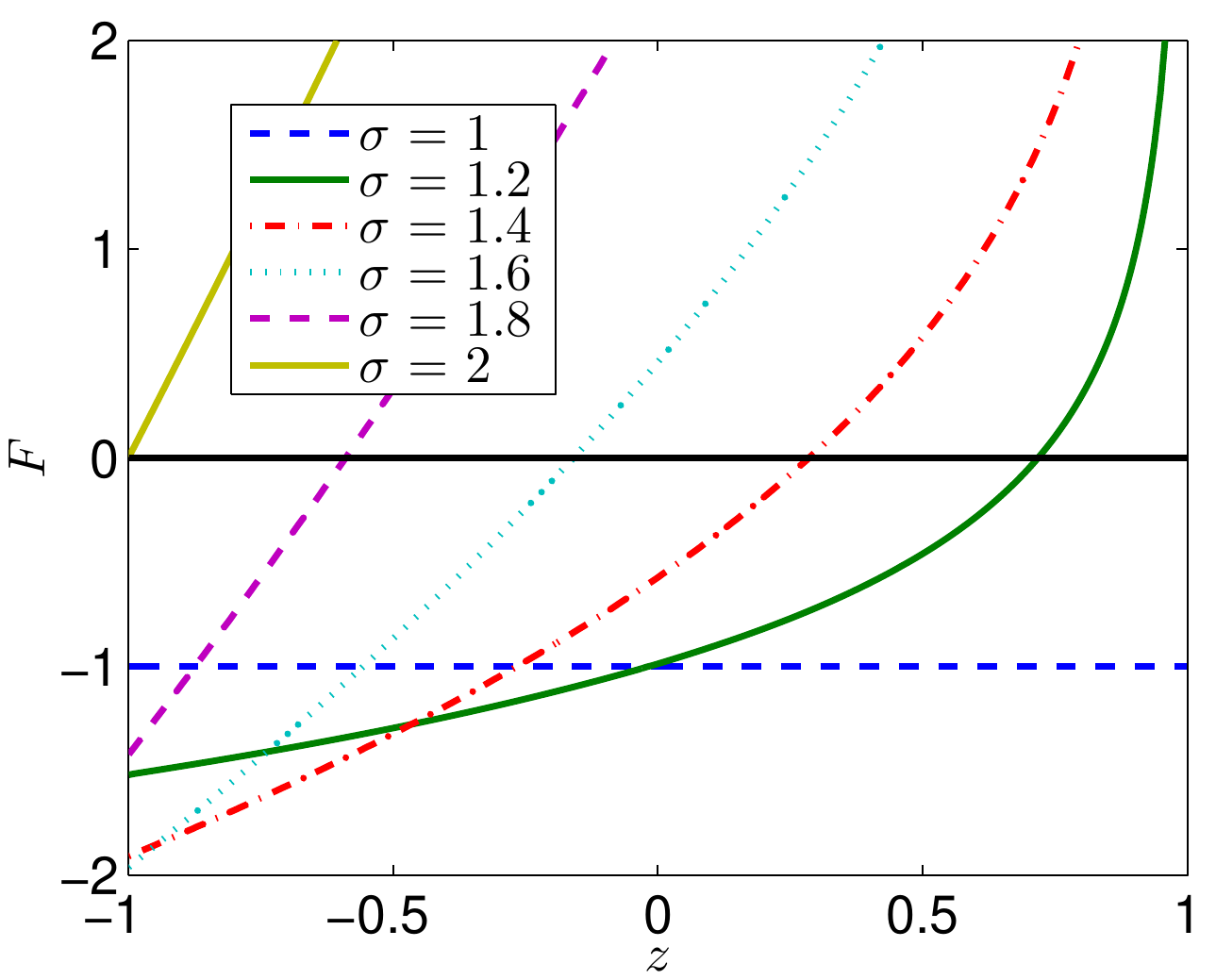}
\caption{$z_0$ is positive for $\sigma$ near $1,$ and becomes negative as $\sigma$ is close to $2$. Thus, the condition $z_0\in\sts{0,1}$ means that $\sigma$ can not be close to $2$. This figure comes from \cite{LiuSS:gDNLS:Stab}.}\label{fig:Fcurves}
\end{figure}

   \item \textit{ The technical assumption}{:}
    $\frac{\omega^0_2-\omega^0_1}{c^0_2-c^0_1}>0$. Because of the fact that the radiation term cannot separate from the solitary wave along the flow \eqref{gDNLS}, this technical assumption allows us to deal with some "bad" term with good sign in \eqref{3}, see the monotonicity formulas in Section \ref{sec:mono} for more details.

   \item \textit{ The relative speed assumption}: In fact, it is sufficient that $c^0_1
    <
    2\frac{\omega^0_2-\omega^0_1}{c^0_2-c^0_1}
    <c^0_2
    $ from our proof. However,
    we suppose
    $c^0_1 <
    \frac{\omega^0_2-\omega^0_1}{c^0_2-c^0_1}$, and
    $4\frac{\omega^0_2-\omega^0_1}{c^0_2-c^0_1}<c^0_2$ for the convenience.
    In addition, combining $4\frac{\omega^0_2-\omega^0_1}{c^0_2-c^0_1}<c^0_2$
    with $\frac{\omega^0_2-\omega^0_1}{c^0_2-c^0_1}>0,$ we immediately obtains that
    $ c_2^0 $ and $z_0(\sigma)$ need to be positive.

    \item
The stability of the sum of two solitary waves can be easily extended to that of the $k$ solitary
waves case, $k\geq 3$.
\end{enumerate}
\end{remark}

At last, the paper is organized as following.  In Section
\ref{sect:PreRes},  we introduce the linearized operator around
the solitary wave, and show the coercivity property of the
linearized operator under the geometric constraints; In Section \ref{sect:mod}, we
give the modulation analysis of the solution around the sum of two
solitary waves with weak interactions. In Section \ref{sect:mon}, we introduce some extra monotonicity formulas and their
variance along the flow \eqref{gDNLS}. In Section \ref{sect:proof},  we firstly introduce
a localized action functional, which is almost conserved by the
monotonicity formula and the conservation laws of mass, momentum and
energy, to refine the energy estimate about the radiation term in the
modulation analysis of the solution; secondly, we use some
monotonicity formulas to refine the estimates of the parameter variance
$|\omega_k(t)-\omega_k(0)|+|c_k(t)-c_k(0)|$, $k=1, 2$ besides of the
conservation laws of mass and momentum. These refined estimates improve
the energy estimate of the radiation term in the modulation analysis and imply Theorem \ref{mainthm} together with the bootstrap argument in \cite{MartelMT:Stab:gKdV, MartelMT:Stab:NLS} (see also \cite{LeWu:DNLS, MiaoTX:DNLS:Stab}). In Appendix A, for all solitary waves $Q_k^0(x)=Q_{\omega^0_k,
    c^0_k}(x)$, $k=1, 2,$ which satisfy the conditions in Theorem \ref{mainthm}, we verify the fact that
\begin{align*}
    \ltl{ 2M\sts{Q_k^0}\norm{ \partial_x Q_k^0}_{L^2}^2 - 4\left[P\sts{ {Q}_k^0 } \right]^2} \not = 0,
\end{align*}
which is used to show the non-degenerate condition \eqref{nondeg:determ}. In Appendix B, we give the expansion of the action functional $\mathcal{S}(t)$ (i.e., Lemma \ref{lem:tlexp}) in details.

\section{Preliminary results}\label{sect:PreRes}
In this section, we give some basic facts about the solitary waves for \eqref{gDNLS}. Let $(\omega, c)\in \R^2$ with $4\omega>c^2$, and $u\sts{t,x}=\varphi_{\omega,c}\sts{x-ct}\e^{\i\omega t}$ be a solution of \eqref{gDNLS}, it is easy to check that $\varphi_{\omega,c}$ satisfies
\begin{align}\label{sltn:eq}
   \omega {\varphi_{\omega,c}}-\partial^2_x\varphi_{\omega,c} +\i c \partial_x{\varphi_{\omega,c}} - \i\abs{{\varphi_{\omega,c}} }^{2\sigma} {\partial_x \varphi_{\omega,c}}=0.
\end{align}
Now define the set $ \mathcal G_{\omega,c}$ of the solitary waves to \eqref{gDNLS}
\[
    \mathcal G_{\omega,c}
    =
    \{\varphi_{\omega,c}\in H^1\sts{\R}\setminus\ltl{0} : \varphi_{\omega,c} ~\text{satisfies}~ \eqref{sltn:eq}\}.
\]
and let
\begin{align}
\label{Q}
  Q_{\omega,c}\sts{x}=\Phi_{\omega,c}(x)\exp \left\{
    \i \frac{c}{2}x-\frac{\i}{2\sigma +
      2}\int^{x}_{-\infty}\Phi_{\omega,c}^{2\sigma}(y)\d y\right\},
\end{align}
with
\begin{equation}
  \label{phi}
  \Phi_{\omega,c}(x)
  =
  \sts{
    \frac{(\sigma+1)(4\omega-c^2)}{
        2\sqrt{\omega}(\cosh(\sigma\sqrt{4\omega-c^2}x)-\frac{c}{2\sqrt{\omega}})
        }
    }^{\frac{1}{2\sigma}}.
\end{equation}

The first result is the variational characterization lemma of the solitary waves.
\begin{lemma}[Variational characterization of solitary waves \cite{FuHaIn:gDNLS:GWP}]
\label{lem:variation}
Suppose $\sts{\omega,c}\in \R^2$ satisfies $4\omega>c^2$. Let $d\sts{\omega,c}$, $S_{\omega,c}$ and $K_{\omega,c}$ be defined by \eqref{d}, \eqref{S} and \eqref{K} respectively. Then we have
\begin{align*}
\mathcal G_{\omega,c}= & ~ \left\{ \varphi\in H^1(\R)
\setminus\ltl{0}:S_{\omega,c}(\varphi)=  d\sts{\omega,c},\,K_{\omega,c}(\varphi)=0\, \right\} \\
= & ~ \left\{Q_{\omega,c}(\cdot-y)\e^{\i\theta} : \theta \in [0,2\pi), y\in \R \right\}.
\end{align*}
\end{lemma}
\begin{remark}
  \begin{enumerate}
    \item  By the Lagrange multiplier argument in \cite{FuHaIn:gDNLS:GWP}, we have
    $S_{\omega,c}'(Q_{\omega,c})=0,$ which implies that
    $ d'\sts{\omega,c} = \sts{M\sts{Q_{\omega,c}},P\sts{Q_{\omega,c}}}$
    and
    \begin{align}\label{hess}
      d''\sts{\omega,c}
      =
      \begin{pmatrix}
        \partial_{\omega}M\sts{Q_{\omega,c}} & \partial_{c}M\sts{Q_{\omega,c}} \\
        \partial_{\omega}P\sts{Q_{\omega,c}} & \partial_{c}P\sts{Q_{\omega,c}} \\
      \end{pmatrix}.
    \end{align}

    \item  By the explicit formula of the solitary waves, the following nondegenerate condition
        \begin{align*}
          \det\left[ d''\sts{\omega,c} \right]<0
        \end{align*}
       holds with $\sigma\in\sts{1,2}$ and
        $c\in\sts{ -2\sqrt{\omega} ,2z_0\sqrt{\omega}}$,  see Theorem 4.3 in \cite{LiuSS:gDNLS:Stab}. This nondegenerate condition is important to show the stability result of the solitary waves by the perturbation argument, the modulation stability and the energy method.
  \end{enumerate}
\end{remark}
\begin{proposition}[Coercivity property of the linearized operator]
  \label{prop:coercivity}
  Let $\sigma$, $z_0$ be as that in Theorem \ref{mainthm}, and $(\omega, c)\in\R^2$ with
  $c\in\sts{ -2\sqrt{\omega} ,2z_0\sqrt{\omega}}.$
  If $\eps\in H^1(\R)$ satisfies the orthogonality conditions
\begin{align}\label{spect:orth}
  \bilin{\eps}{\i Q_{\omega,c}}
  =
  \bilin{\eps}{\partial_x{Q}_{\omega,c}}
  =
  \bilin{\eps}{Q_{\omega,c}}
  =
  \bilin{\eps}{\i\partial_x{Q}_{\omega,c}}
  =
  0,
\end{align}
  then we have
  \[
  \biact{S_{\omega,c}''\sts{Q_{\omega,c}}}{\eps}{\eps}
  \geq C_{\mathrm{abs}} \norm{\eps}_{H^1}^2,
  \]
where
\begin{align}\label{quad}
 S_{\omega,c}''(Q_{\omega,c}) := T_{\omega,c} + N_{\omega,c}
\end{align}
with $\dual{T_{\omega,c}\eps}{\eps}
  :=
  \int
  \sts{
    \abs{\eps_x}^2+\omega~\abs{\eps}^2
   -c\Im\sts{ \bar{\eps}\eps_x } },$ and \begin{align*}
    \dual{N_{\omega,c}\eps}{\eps}
    :=
    &
      \Im\int\left[ \abs{Q_{\omega,c}}^{2\sigma}\bar{\eps}\eps_x
      +
      \sigma
         \abs{Q_{\omega,c}}^{2\sigma-2}
         \left(
            \bar{Q}_{\omega,c}\partial_x{Q}_{\omega,c}\abs{\eps}^2
            +
            Q_{\omega,c}\partial_x{Q}_{\omega,c}\bar{\eps}^2
         \right)
        \right].
\end{align*}
\end{proposition}
\begin{proof} We follow the argument in \cite{LeWu:DNLS,MiaoTX:DNLS:Stab} (see also \cite{MartelMT:Stab:NLS,Wein:stab:SJMA, Wein:stab:CPAM}) and the references therein, and divide the proof into several steps.
\begin{enumerate}[label=\emph{{Step \arabic*.}},ref=\emph{{Step \arabic*}}]
\item\label{step:spectral}\textit{Spectral distribution of $ S_{\omega,c}''(Q_{\omega,c}) $.}
On one hand, by H\"{o}lder's inequality, we have
\begin{align*}
  \dual{T_{\omega,c}\eps}{\eps}
  =
  &
  \Re\int
    \ntn{
        \abs{\partial_x\eps}^2+\i c \partial_x\eps\bar{\eps}+\omega \abs{\eps}^2
    }
  \\
  =
  &
  \Re\int
    \ntn{
        \abs{\partial_x\eps}^2+\i c~\partial_x\eps\bar{\eps}+\frac{c^2}{4}~\abs{\eps}^2
    }
   +
   \sts{\omega-\frac{c^2}{4}}\int\abs{\eps}^2
     \\
  \geq
  &
   \sts{\omega-\frac{c^2}{4}}\int\abs{\eps}^2,
\end{align*}
which means that ${\bm\sigma}_{\ess}\sts{T_{\omega,c}}\subset\left[\omega-\frac{c^2}{4}~,~\infty\right).$
On the other hand, by the exponential decay of $ Q_{\omega,c} $ and the similar argument of Proposition 2.9 in \cite{Wein:stab:SJMA}, we know that the operator $N_{\omega,c}$ is relatively compact with respect to $T_{\omega,c}$. By Weyl's theorem in \cite{ReedSimon:book:IV}, we have
\begin{align}\label{ess-spectral}
  {\bm\sigma}_{\ess}\sts{ S_{\omega,c}''(Q_{\omega,c}) } =  {\bm\sigma}_{\ess}\sts{T_{\omega,c}} \subset\left[\omega-\frac{c^2}{4}~,~\infty\right).
\end{align}

\item\label{step:w-mp}
\textit{
    We claim that for any $\varphi\in H^1\sts{\R}\setminus\ltl{0}$ with
    $\act{ K'_{\omega,c}\sts{Q_{\omega,c}} }{\varphi}=0,$
      \begin{align}\label{coer-mp}
        \biact{ S''_{\omega,c}\sts{Q_{\omega,c}} }{ \varphi }{ \varphi }\geq0.
      \end{align}
    }
In deed,  notice that $ K'_{\omega,c}\sts{Q_{\omega,c}}\neq 0 $, we can choose $\psi$
   such that $\act{ K'_{\omega,c}\sts{Q_{\omega,c}} }{ \psi }\neq 0$. We now define
   for any $ \varphi \in H^1\sts{\R}\setminus\ltl{0} $ with $ \act{ K'_{\omega,c}\sts{Q_{\omega,c} } }{\varphi}=0 $,~
\begin{align*}
 \kappa\sts{m,s}:=K_{\omega,c}\sts{Q_{\omega,c} +m \psi+s\varphi}.
\end{align*}
   Applying the Implicit Function Theorem to $\kappa\sts{m,s}$ with
   $
     \kappa\sts{0,0}= K_{\omega,c}\sts{Q_{\omega,c}}
   $
   and
   $
     \left.
        \partial_{m}\kappa\sts{m,s}
     \right|_{ \sts{m,s}=\sts{0,0} }
     \neq 0
   $
    yields that there exists $\delta>0$ such that
    $m:\sts{-\delta,\delta}\mapsto\R$ is of class $\mathcal{C}^1$ with $m\sts{0}=0$,
    and
   \begin{align}
   \label{s1}
      \kappa\sts{m\sts{s}, s}=K_{\omega,c}\sts{Q_{\omega,c} +m\sts{s}\psi+s\varphi}\equiv0\text{~~for~~}s\in\sts{-\delta,\delta}.
   \end{align}
Differentiating on $s$, we have
$$
    \dot{m}\sts{0}\partial_{m}\kappa\sts{0,0} + \partial_s\kappa\sts{0,0}=0,
$$
   which implies that
$
    \dual{ K'_{\omega,c}\sts{Q_{\omega,c}} }{ \varphi }
        +
    \dot{m}\sts{0}\dual{ K'_{\omega,c}\sts{Q_{\omega,c}} }{ \psi }=0.
$
Consequently we have
$
    \dot{m}\sts{0}=0.
$

Based on the above argument, we can define the function $\iota: \sts{-\delta,\delta} \mapsto\R $
as following:
\begin{align*}
    \iota\sts{s}:=S_{\omega,c}\sts{ Q_{\omega,c}+m \sts{s}\psi+s\varphi }.
\end{align*}
It means from Lemma \ref{lem:variation} and \eqref{s1} that
$0$ is a local minimum point of $\iota,$
and implies that the function $\iota$ is convex around $0$,
i.e.
  $ \iota''\sts{0}=\biact{ S''_{\omega,c}\sts{ Q_{\omega,c} } }{ \varphi }{ \varphi }\geq0.  $

\item\label{step:negtfuncs}
\textit{
$ S_{\omega,c}''(Q_{\omega,c}) $ has at least one negative eigenfunction.}
For this purpose, we only need to show that there exists a function $U$ in $H^1$ with
$\biact{ S_{\omega,c}''(Q_{\omega,c}) }{U}{U}<0.$ Indeed, it follows from
$ K_{\omega,c}(Q_{\omega,c})=0 $ and $4\omega>c^2$ that
\begin{align}
\label{S-negative}
    \biact{ S''_{\omega,c}(Q_{\omega,c}) }{ Q_{\omega,c} }{ Q_{\omega,c} }
    =
   &
    -
    2\sigma
    \Re\int
    \ltl{ \abs{\partial_x{Q_{\omega,c}}}^2 + \omega\abs{Q_{\omega,c}}
    + \i c~\partial_x{Q_{\omega,c}} \bar{Q}_{\omega,c}
    }
    <0.
\end{align}

\item\label{step:uniqnegtfunc}\textit{
$ S_{\omega,c}''(Q_{\omega,c}) $ has at most one-dimensional negative eigenspace.}
We argue by contradiction. Suppose that there exist two linearly independent eigenfunctions
$\chi_1$ and~$\chi_2$ of $ S_{\omega,c}''(Q_{\omega,c}).$ Since $ S_{\omega,c}''(Q_{\omega,c}) $ is a self-adjoint operator, without of generality, one may assume that
$ \inprod{\chi_1}{\chi_2}=0.$ It is easy to check that
$$ \biact{  S_{\omega,c}''(Q_{\omega,c}) }{ \chi_1 }{ \chi_2 }=0.$$
Moreover by the nonnegative property in \ref{step:w-mp} and $ \biact{  S_{\omega,c}''(Q_{\omega,c}) }{ \chi_1 }{ \chi_1 }<0 $
and
$ \biact{  S_{\omega,c}''(Q_{\omega,c}) }{ \chi_2 }{ \chi_2 }<0 $
, we have
$$
  \act{ K'_{\omega,c}\sts{Q_{\omega,c}} }{\chi_1}\neq0
, \;\text{and}\;
    \act{ K'_{\omega,c}\sts{Q_{\omega,c}} }{\chi_2}\neq0,
$$
which implies that there exists $\xi_0\in\R\setminus\ltl{0}$ with $ \chi_0=\chi_1+\xi_0\chi_2 $ such that
\begin{align}\label{s3}
  \act{ K'_{\omega,c}\sts{Q_{\omega,c}} }{ \chi_0 }=0.
\end{align}
By the nonnegative property in \ref{step:w-mp}, we have
$
  \biact{  S_{\omega,c}''(Q_{\omega,c}) }{ \chi_0 }{ \chi_0 }\geq0,
$
which is in contradiction with
\begin{align*}
  \biact{  S_{\omega,c}''(Q_{\omega,c}) }{ \chi_0 }{ \chi_0 }
  =
  &
    \biact{  S_{\omega,c}''(Q_{\omega,c}) }{ \chi_1 }{ \chi_1 }
    +
    \xi_0^2\biact{  S_{\omega,c}''(Q_{\omega,c}) }{ \chi_2 }{ \chi_2 }
    <0.
\end{align*}\

\item\label{step:ker}
    $ \ker\sts{ S_{\omega,c}''(Q_{\omega,c}) }=\spn\{ \i Q_{\omega,c}, \partial_x{ Q_{\omega,c} }   \}. $
It follows from Proposition 3.6 in
\cite{LiuSS:gDNLS:Stab}.

\item\label{step:posit}
\textit{Positivity of the quadratic form $ \dual{S_{\omega,c}''(Q_{\omega,c})\eps}{\eps}.$}
In fact, we have
\begin{lemma}
\label{cla:posit}
  For any $\eps\in H^1\sts{\R}\setminus\ltl{0}$ with \eqref{spect:orth}
we have
$
        \biact{S_{\omega,c}''\sts{Q_{\omega,c}}}{\eps}{\eps}>0.
$
\end{lemma}
\begin{proof} This is a consequence of
\ref{step:spectral}--\ref{step:ker} and the standard spectral decomposition arguments for the quadratic form
$\biact{S_{\omega,c}''\sts{Q_{\omega,c}}}{\eps}{\eps}.$
In this proof, we will ignore the subscript $\omega $ and $c$ for convenience and write $S_{\omega,c}''\sts{Q_{\omega,c}}$ and $Q_{\omega,c}$
as $S''\sts{Q}$ and $Q$ respectively.

First, we infer, from \ref{step:negtfuncs}--\ref{step:ker} together with \eqref{ess-spectral}, that
the space $H^1$ can be decomposed as a direct sum of three subspaces:
\begin{align}\label{space:decom}
  H^1={\rm N}\bigoplus{\rm K}\bigoplus{\rm P}~,
\end{align}
with
$
  {\rm K}:=\spn\{ \i Q~,~\partial_x{ Q }  \}
$,
$
  {\rm P}:=\set{\eps}{ \biact{S''\sts{Q}}{\eps}{\eps}>0 }
$
and
$
  {\rm N}:=\spn\ltl{ \chi },
$
where $\chi$ is the $L^2$-normalized negative eigenfunction
corresponding to the negative eigenvalue $ -\lambda^2$.
According to \eqref{space:decom}, we can decompose any function $\eps\in H^1$
satisfying \eqref{spect:orth} into
\begin{align}\label{s6:6}
  \eps= \kappa\chi + \p
\end{align}
with $\p\in{\rm P}$ and $\kappa = \inprod{\eps}{\chi}.$

Now, we turn to the decomposition of some special functions related to the non-degenerate condition
$\det \ntn{ d''\sts{\omega,c} } <0$.
On one hand, by \eqref{sltn:eq}, we have
$$
  S''\sts{Q} \partial_\omega{Q} = - Q, \;\text{and} \;
  S''\sts{Q} \partial_c{Q} = -\i Q,
$$
which implies that
\begin{align}
\label{deg1}
&
  \biact{ S''\sts{Q} }{ \partial_\omega{Q} }{ \partial_\omega{Q} }
  =
  -\partial_\omega M\sts{Q},
&
  \biact{ S''\sts{Q} }{ \partial_\omega{Q} }{ \partial_c{Q} }
  =
  -\partial_c M\sts{Q},
  \\
\label{deg2}
&  \biact{ S''\sts{Q} }{ \partial_c{Q} }{ \partial_\omega{Q} }
  =
  -\partial_\omega P\sts{Q},
&
  \biact{ S''\sts{Q} }{ \partial_c{Q} }{ \partial_c{Q} }
  =
  -\partial_c P\sts{Q},
\end{align}
and
\begin{align}\label{s6:2}
  \biact{ S''\sts{Q} }{ \partial_\omega{Q} }{ \eps } = \biact{ S''\sts{Q} }{ \partial_c{Q} }{ \eps }=0.
\end{align}
On the other hand, the non-degenerate condition $\det\ntn{  d''\sts{\omega,c} }<0$ implies that
there exists ${\bm\xi}=(\xi_1, \xi_2)\in\R^2$ such that $ \biact{ d''\sts{\omega,c} }{\bm\xi}{\bm\xi}<0, $
which together with \eqref{hess}, \eqref{deg1}-\eqref{deg2} and setting
$ U = \xi_1\partial_\omega{Q} + \xi_2\partial_c{Q} $ yields that
\begin{align}
\label{s6:1}
    \biact{ S''\sts{Q} }
    { U }
    { U }
  <0.
\end{align}
Using the decomposition \eqref{space:decom}, we decompose the function $U$ as following:
\begin{align}
\label{s6:5}
  U :=\alpha \chi + \zeta + {\mathtt y},
\end{align}
with $\alpha = \inprod{U}{\chi},$ $\zeta\in \ker\sts{ S''(Q) } $ and ${\mathtt y }\in {\rm P}.$
From \eqref{s6:6}, \eqref{s6:2}-\eqref{s6:5}, we have
\begin{align}
\label{s6:3}
  0>\biact{ S''\sts{Q} }{ U }{ U }
  &
  =
  -\alpha^2\lambda^2 + \biact{ S''\sts{Q} }{ {\mathtt y} }{ {\mathtt y} },
\\
\label{s6:4}
  0=\biact{ S''\sts{Q} }{ U }{ \eps }
  &
  =
  -\alpha\kappa\lambda^2 + \biact{ S''\sts{Q} }{ {\mathtt y} }{ \p }.
\end{align}
Now inserting \eqref{s6:6} into
$\biact{S''\sts{Q}}{\eps}{\eps}$, and taking into account \eqref{s6:3}-\eqref{s6:4},
we obtain by the Cauchy--Schwarz inequality that
\begin{align*}
  \biact{S''\sts{Q}}{\eps}{\eps}
  =
    -\kappa^2\lambda^2
    +
    \biact{S''\sts{Q}}{ \p }{ \p }
  \geq
  -\kappa^2\lambda^2
  +
  \frac
    { \biact{S''\sts{Q}}{ {\mathtt y} }{ \p }^2}
    { \biact{S''\sts{Q}}{ {\mathtt y} }{ {\mathtt y} } }
  >
  -\kappa^2\lambda^2
  +
  \frac
    { \alpha^2\kappa^2\lambda^4 }
    { \alpha^2\lambda^2 }
  =
  0.
\end{align*}
This completes the proof.
\end{proof}
\end{enumerate}

By \ref{step:spectral} to \ref{step:posit},  the coercivity property of $ \dual{S_{\omega,c}''(Q_{\omega,c})\eps}{\eps}$ can be obtained by the argument in
  \cite{LeWu:DNLS} and \cite{MiaoTX:DNLS:Stab} (see also \cite{MartelMT:Stab:NLS,Wein:stab:SJMA, Wein:stab:CPAM}). This concludes the proof of the proposition.
  \end{proof}

\section{Modulation Analysis}\label{sect:mod}
Following the modulation analysis in \cite{LeWu:DNLS} \cite{MiaoTX:DNLS:Stab} (also \cite{MartelMT:Stab:gKdV, MartelMT:Stab:NLS, Wein:stab:SJMA, Wein:stab:CPAM}), we will show the geometrical decomposition of the solutions to \eqref{gDNLS} close to the sum of two solitary waves with weak interactions. Now let $(\sigma, z_0)$ be as in Theorem \ref{mainthm}, $(\omega_{j}^{0},
c_{j}^{0})\in \R^2$ be such that $-2 \sqrt{\omega_{j}^{0}} <c_{j}^{0}
<2z_0\sqrt{\omega_{j}^{0}}$, $j=1, 2$, then by Theorem 4.3 in \cite{LiuPS:DNLS:ISM}, we have the non-degenerate condition
        \begin{align}\label{nondeg}
          \det\ntn{ d''\sts{\omega^0_j,c^0_j} }  <0, \;\text{for}\; j=1, 2.
        \end{align}
Let $\alpha < \alpha_0$ be small
enough, and $L>L_0$ be large enough, where $\alpha_0$, $L_0$ will be
determined later. We first consider the tube of size $\alpha$ in the energy space $H^1(\R)$
\begin{align*}
  \U{\alpha}{\bm{\omega}^0}{{\bf c}^0}{L}
  :=
  \set{u\in H^1\sts{\R}\setminus\ltl{0} }{
    \inf_{\substack{x_2-x_1>L,\\ \gamma_1,\gamma_2\in\R} }
    \norm{ u-\sum_{j=1}^{2} Q_{\omega_j^0,c_j^0}\sts{\cdot-x_j}\e^{\i\gamma_j} }_{H^1}<\alpha }
\end{align*}
with
$ {\bm\omega}^0=\sts{~\omega_1^0~,~\omega_2^0~} $ and $ {\mathbf{c}}^0=\sts{~c_1^0~,~c_2^0~}.$
We denote
$Q_j^0=Q_{\omega_j^0,c_j^0},~Q_j=Q_{\omega_j,c_j}$ for convenience,
and let
$\bm{\omega}$, $\mathbf{c}$, $\mathbf{x}$ and $\bm{\gamma}$
be the vectors $\sts{\omega_1, \omega_2},$
$\sts{c_1, c_2},$ $\sts{x_1, x_2}$
and $\sts{\gamma_1, \gamma_2}$ respectively.

By the Implicit Function Theorem, we have

\begin{lemma}[Static version]\label{lem:mod1}There exist $ L_{\ift} $ large enough, $\alpha_{\ift}$ small enough,
such that for any $ L>L_{\ift} ,$~$\alpha<\alpha_{\ift}$, if $u\in \U{\alpha}{\bm{\omega}^0}{{\bf c}^0}{L},$
then there exist unique $\mathcal{C}^1$ functions $ \bm{\omega}, \mathbf{c}, \mathbf{x}, \bm{\gamma}$ such that the following decomposition holds:
\begin{align}\label{funct:decom}
  u\sts{x}=\sum_{j=1}^{2} Q_j\sts{x-x_j}\e^{\i\gamma_j} + \eps\sts{x },
\end{align}
with $-2 \sqrt{\omega_{j}} <c_{j}
<2z_0\sqrt{\omega_{j}}$, $j=1,2$ and
\begin{align}
\label{orth:stat}
  \inprod{\eps}{ R_j }=\inprod{\eps}{\i\partial_x{R_j} }=\inprod{\eps}{\i R_j}=\inprod{ \eps }{ \partial_x{R_j} }=0,
  \quad j=1,2,
\end{align}
where
$
  R_j\sts{x}=Q_j\sts{x-x_j}\e^{\i\gamma_j}.
$
Moreover, we have
\begin{gather}
\label{est:ift}
  \norm{\eps}_{H^1}+\sum^{2}_{j=1}\left(\abs{ \omega_j-\omega_j^0 } + \abs{ c_j-c_j^0 }\right)
  <C_{\ift}\alpha,
  \quad j=1,2,\\
   x_2-x_1 >\frac{L}{2}.
\end{gather}
and
\begin{align}
\label{tech:nondeg}
    \frac{1}{2}<\frac{ \sqrt{ 4\omega_j-\sts{c_j}^2 } }{ \sqrt{ 4\omega_j^0-\sts{c_j^0}^2 } }<2,\quad j=1,2,
\end{align}

\end{lemma}
\begin{proof}First of all, by the definition of  $ \U{\alpha}{\bm{\omega}^0}{{\bf c}^0}{L},$
there exist $\mathbf{x}^0:=\sts{x_1^0, x_2^0 }\in \R^2$ with $x_2^0 -x_1^0\geq L  $
and ${\bm \gamma}^0:=\sts{\gamma_1^0,\gamma_2^0 } \in \R^2$ such that
\begin{align}
\label{eq:h1alpha}
  \norm{ u-\sum_{j=1}^{2} Q_{\omega_j^0,c_j^0}\sts{\cdot-x_j^0}\e^{\i\gamma_j^0} }_{H^1}<\alpha.
\end{align}
Let  $\q = \sts{
    \bm{\omega}, \mathbf{c}, \mathbf{x}, \bm{\gamma} }
$ and
\begin{align*}
\q^0 =    \sts{
    \bm{\omega}^0, \mathbf{c}^0, \mathbf{x}^0, \bm{\gamma}^0 },\quad \mathbf{Q}^0 (x)= & \sum_{j=1}^{2} Q_{\omega_j^0,c_j^0}\sts{x-x_j^0}\e^{\i\gamma_j^0}.
\end{align*}
For any $ u $ with \eqref{eq:h1alpha} and $\q $, we define
\begin{align}\label{funct:decom}
  \eps\sts{x; \q, u }:=u\sts{x}-\sum_{j=1}^{2} Q_j\sts{x-x_j}\e^{\i\gamma_j}.
\end{align}
It is easy to see that
\begin{align}
\label{eq:2epsilon0}
    \varepsilon
    \sts{x; \q^{0}, \mathbf{Q}^0
    }\equiv 0.
\end{align}

Defining
$ {\rm P}\sts{\q, u}
    :=
    \sts{
        \varrho_1^1,~\varrho_1^2,~\varrho_1^3,~\varrho_1^4,~
        \varrho_2^1,~\varrho_2^2,~\varrho_2^3,~\varrho_2^4
    }\sts{\q, u }
$
by
  \begin{align*}
    &
    \varrho_j^1\sts{\q,u }
        :=\inprod{ \eps\sts{\cdot~;\q,u } }{~Q_j\sts{\cdot-x_j}\e^{\i\gamma_j}},\\
    &
    \varrho_j^2\sts{\q,u}
        :=\inprod{ \eps\sts{\cdot~;\q,u } }{\i \partial_x{Q_j\sts{\cdot-x_j}\e^{\i\gamma_j}} },
    \\
    &
    \varrho_j^3\sts{\q,u}
        :=\inprod{ \eps\sts{\cdot~;\q,u } }{\i {Q_j\sts{\cdot-x_j}\e^{\i\gamma_j}}},\\
    &
    \varrho_j^4\sts{\q,u }
        :=\inprod{ \eps\sts{\cdot~;\q,u } }{~\partial_x {Q_j\sts{\cdot-x_j}\e^{\i\gamma_j}}},
  \end{align*}
 where $k=1, 2$.  By simple calculations, we have
  \begin{align}
  \label{eq:eps2}
    \frac{\partial\eps}{\partial \omega_j}= -\partial_{\omega_j}{Q_j}\sts{x-x_j}\e^{\i\gamma_j},
    &\quad \frac{\partial\eps}{\partial c_j}=-\partial_{c_j}{Q_j}\sts{x-x_j}\e^{\i\gamma_j},\\
    \label{eq:eps1}
    \frac{\partial\eps}{\partial x_j}=~~\partial_x{Q_j}\sts{x-x_j}\e^{\i\gamma_j},
    &\quad \frac{\partial\eps}{\partial \gamma_j}=~~-\i{Q_j}\sts{x-x_j}\e^{\i\gamma_j},
  \end{align}
and
  \begin{align}
  \label{26}
    \int \abs{ \mathcal{Q}_1^0\sts{ x-x_1^0 }\e^{\i\gamma_1^0} ~ \mathcal{Q}_2^0\sts{ x-x_2^0 }\e^{\i\gamma_2^0} }
    \leq C_{\mathrm{abs}}\e^{-2\theta_1 L} ,
  \end{align}
  where
  $ \theta_1=\min\ltl{ ~\frac{ \sqrt{4\omega_1^0-\sts{c_1^0}^2} }{8}~,~\frac{ \sqrt{4\omega_2^0-\sts{c_2^0}^2} }{8}~ },$
  and
  $\mathcal{Q}_j^0$ denotes one of
  $
    Q_j^0,~\partial_xQ_j^0,~
    \left.\partial_{\omega_j}{Q_j}\right|_{\q=\q^0},
  $
  and
  $
    \left.\partial_{c_j}{Q_j}\right|_{\q=\q^0}.
  $
Inserting \eqref{eq:eps2} and \eqref{eq:eps1} into $\varrho_j^k\sts{\q}, $ we obtain
  \begin{align*}
    &
 \frac{\partial\varrho_j^1}{\partial \omega_k}\sts{\q^0,\mathbf{Q}^0}
    =
    \begin{cases}
      -\frac{\partial }{ \partial\omega_k^0 }M\sts{Q_k^0},&\mbox{if~} j=k,
      \\
      \bigo{ \e^{-2\theta_1 L} },&\mbox{if~} j\neq k,
    \end{cases}
    &
\frac{\partial\varrho_j^1}{\partial c_k}\sts{\q^0,\mathbf{Q}^0}
    =
    \begin{cases}
      -\frac{\partial }{ \partial c_k^0 }M\sts{Q_k^0},&\mbox{if~} j=k,
      \\
      \bigo{ \e^{-2\theta_1 L} }&\mbox{if~} j\neq k,
    \end{cases}
    \\
    &
\frac{\partial\varrho_j^1}{\partial x_k}\sts{\q^0,\mathbf{Q}^0}
    =
    \begin{cases}
      0,&\mbox{if~} j=k,
      \\
      \bigo{ \e^{ -2\theta_1 L } },&\mbox{if~} j\neq k,
    \end{cases}
    &
\frac{\partial\varrho_j^1}{\partial \gamma_k}\sts{\q^0,\mathbf{Q}^0}
    =
    \begin{cases}
      0,&\mbox{if~} j=k,
      \\
      \bigo{ \e^{ -2\theta_1 L } },&\mbox{if~} j\neq k,
    \end{cases}
  \end{align*}
  \begin{align*}
    &
\frac{\partial\varrho_j^2}{\partial \omega_k}\sts{\q^0,\mathbf{Q}^0}
    =
    \begin{cases}
      -\frac{\partial }{ \partial\omega_k^0 }P\sts{Q_k^0},&\mbox{~if~} j=k,
      \\
      \bigo{ \e^{-2\theta_1 L} },&\mbox{~if~} j\neq k,
    \end{cases}
    &
\frac{\partial\varrho_j^2}{\partial c_k}\sts{\q^0,\mathbf{Q}^0}
    =
    \begin{cases}
      -\frac{\partial }{ \partial c_k^0 }P\sts{Q_k^0},&\mbox{~if~} j=k,
      \\
      \bigo{ \e^{-2\theta_1 L} },&\mbox{~if~} j\neq k,
    \end{cases}
    \\
    &
\frac{\partial\varrho_j^2}{\partial x_k}\sts{\q^0,\mathbf{Q}^0}
    =
    \begin{cases}
      0,&\mbox{~if~} j=k,
      \\
      \bigo{ \e^{ -2\theta_1 L } },&\mbox{~if~} j\neq k,
    \end{cases}
    &
\frac{\partial\varrho_j^2}{\partial \gamma_k}\sts{\q^0,\mathbf{Q}^0}
    =
    \begin{cases}
      0,&\mbox{~if~} j=k,
      \\
      \bigo{ \e^{ -2\theta_1 L } },&\mbox{~if~} j\neq k,
    \end{cases}
  \end{align*}
  \begin{align*}
    &
\frac{\partial\varrho_j^3}{\partial \omega_k}\sts{\q^0,\mathbf{Q}^0}
    =
    \begin{cases}
      -\Im\int\frac{\partial}{\partial\omega_k^0}Q_k^0\bar{Q}_k^0,&\mbox{~if~} j=k,
      \\
      \bigo{ \e^{-2\theta_2 L} },&\mbox{~if~} j\neq k,
    \end{cases}
    &&
\frac{\partial\varrho_j^3}{\partial c_k}\sts{\q^0,\mathbf{Q}^0}
    =
    \begin{cases}
      -\Im\int\frac{\partial}{\partial c_k^0}Q_k^0\bar{Q}_k^0,&\mbox{~if~} j=k,
      \\
      \bigo{ \e^{-2\theta_2 L} },&\mbox{~if~} j\neq k,
    \end{cases}
    \\
    &
\frac{\partial\varrho_j^3}{\partial x_k}\sts{\q^0,\mathbf{Q}^0}
    =
    \begin{cases}
      -2P\sts{ {Q}_k^0 },&\mbox{~if~} j=k,
      \\
      \bigo{ \e^{ -2\theta_2 L } },&\mbox{~if~} j\neq k,
    \end{cases}
    &&
\frac{\partial\varrho_j^3}{\partial \gamma_k}\sts{\q^0,\mathbf{Q}^0}
    =
    \begin{cases}
      -2M\sts{ Q_k^0 },&\mbox{~if~} j=k,
      \\
      \bigo{ \e^{ -2\theta_2 L } },&\mbox{~if~} j\neq k,
    \end{cases}
  \end{align*}
  \begin{align*}
    &
\frac{\partial\varrho_j^4}{\partial \omega_k}\sts{\q^0,\mathbf{Q}^0}
    =
    \begin{cases}
      -\Re\int\frac{\partial }{ \partial \omega_k^0 }{Q_k^0}{\partial_x \bar{Q}_k^0},&\mbox{~if~} j=k,
      \\
      \bigo{ \e^{-2\theta_2 L} },&\mbox{~if~} j\neq k,
    \end{cases}
    &&
 \frac{\partial\varrho_j^4}{\partial c_k}\sts{\q^0,\mathbf{Q}^0}
    =
    \begin{cases}
      -\Re\int\frac{\partial }{ \partial c_k^0 }{Q_k^0}{\partial_x \bar{Q}_k^0},&\mbox{~if~} j=k,
      \\
      \bigo{ \e^{-2\theta_2 L} },&\mbox{if~} j\neq k,
    \end{cases}
    \\
    &
\frac{\partial\varrho_j^4}{\partial x_k}\sts{\q^0,\mathbf{Q}^0}
    =
    \begin{cases}
      \norm{ \partial_x Q_k^0}_{2}^2,&\mbox{if~} j=k,
      \\
      \bigo{ \e^{ -2\theta_2 L } },&\mbox{if~} j\neq k,
    \end{cases}
    &&
\frac{\partial\varrho_j^4}{\partial \gamma_k}\sts{\q^0,\mathbf{Q}^0}
    =
    \begin{cases}
      2P\sts{ Q_k^0 },&\mbox{if~} j=k,
      \\
      \bigo{ \e^{ -2\theta_2 L } },&\mbox{if~} j\neq k.
    \end{cases}
  \end{align*}
Hence we can decompose the Jacobian
$\left.\frac{\mathrm{D}{\rm P}}{\mathrm{D}\q} \right|_{\sts{\q,u}=\sts{\q^0,\mathbf{Q}^0}}$
into four $4\times 4$ submatrices,
\begin{align*}
\frac{\mathrm{D}{\rm P}}{\mathrm{D}\q} \sts{\q^0,\mathbf{Q}^0}
    =
   \left. \begin{pmatrix}
      \frac{\mathrm{D}{\rm P}_{1,1}}{\mathrm{D}\,\,\,\,\q\,\,\,}
      &
      \frac{\mathrm{D}{\rm P}_{1,2}}{\mathrm{D}\,\,\,\,\q\,\,\,}
      \\[6pt]
      \frac{\mathrm{D}{\rm P}_{2,1}}{\mathrm{D}\,\,\,\,\q\,\,\,}
      &
     \frac{\mathrm{D}{\rm P}_{2,2}}{\mathrm{D}\,\,\,\,\q\,\,\,}
    \end{pmatrix}\right|_{\sts{\q,u}=\sts{\q^0,\mathbf{Q}^0}}
\end{align*}
where
\begin{align*}
  \left.\frac{\mathrm{D}{\rm P}_{k,k}}{\mathrm{D}\,\,\q\,\,\,} \right|_{\sts{\q,u}=\sts{\q^0,\mathbf{Q}^0}}
  =
  \begin{pmatrix}
    -\frac{\partial }{ \partial \omega_k^0 }M\sts{Q_k^0}
        & -\frac{\partial }{ \partial c_k^0 }M\sts{Q_k^0}
            & 0
                & 0 \\[10pt]
   -\frac{\partial }{ \partial \omega_k^0 }P\sts{Q_k^0}
        & -\frac{\partial }{ \partial c_k^0 }P\sts{Q_k^0}
            & 0
                & 0 \\[10pt]
   \Im\int\frac{\partial}{\partial\omega_k^0}Q_k^0\bar{Q}_k^0
        & \Im\int\frac{\partial}{\partial c_k^0}Q_k^0\bar{Q}_k^0
            & -2P\sts{ {Q}_k^0 }
                & -2M\sts{ Q_k^0 } \\[10pt]
    \Re\int\frac{\partial }{ \partial \omega_k^0 }{Q_k^0}{\partial_x \bar{Q}_k^0}
        & \Re\int\frac{\partial }{ \partial c_k^0 }{Q_k^0}{\partial_x \bar{Q}_k^0}
            & \norm{ \partial_x Q_k^0}_{2}^2
                & 2P\sts{ Q_k^0 }
  \end{pmatrix}.
\end{align*}
By simple calculations, we have
\begin{align*}
  \det\left.\frac{\mathrm{D}{\rm P}_{k,k}}{\mathrm{D}\,\,\q\,\,\,} \right|_{\sts{\q,u}=\sts{\q^0,\mathbf{Q}^0}}
  =
    d''\sts{~\omega_k^0~,~c_k^0~}
    \times
    \ltl{ 2M\sts{Q_k^0}\norm{ \partial_x Q_k^0}_{L^2}^2 - 4\left[P\sts{ {Q}_k^0 } \right]^2},
\end{align*}
and
$$
  \det\left.\frac{\mathrm{D}{\rm P}_{j,k}}{\mathrm{D}\,\,\q\,\,\,} \right|_{\sts{\q,u}=\sts{\q^0,\mathbf{Q}^0}}
  =\bigo{ \e^{ -2\theta_1 L } },
\;\text{ for}\; j\neq k.$$
Putting together, we obtain
\begin{align*}
  \det  \frac{\mathrm{D}{\rm P}}{\mathrm{D}\q} \sts{\q^0,\mathbf{Q}^0}
  =
  \prod_{k=1}^{2} \left\{ \det
    d''\sts{\omega_k^0,c_k^0}
    \times
    \ntn{ 2M\sts{Q_k^0}\norm{ \partial_x Q_k^0}_{L^2}^2 - 4\left[P\sts{ {Q}_k^0 } \right]^2}\right\}
  +
  \bigo{ \e^{ -2\theta_1 L } }.
\end{align*}
The fact that
\begin{align*}
  2M\sts{Q_k^0}\norm{ \partial_x Q_k^0}_{L^2}^2 - 4\left[P\sts{ {Q}_k^0 } \right]^2>0
\end{align*}
in Appendix A, together with the non-degenerate condition \eqref{nondeg} implies that
\begin{align}\label{nondeg:determ}
      \det  \frac{\mathrm{D}{\rm P}}{\mathrm{D}\q} \sts{\q^0,\mathbf{Q}^0}>0
\end{align}
for sufficiently large $L$. We can conclude the proof by the Implicit Function Theorem.
\end{proof}

\begin{lemma}[Dynamic version]
\label{lem:mod2}
Let $L_{\ift}$~and~$\alpha_{\ift} $ be given by
Lemma \ref{lem:mod1}. If $ u\in \mathcal{C}\sts{~\left[0,T^{\ast}\right],H^1} $ is a solution to
\eqref{gDNLS} with $u\sts{0}\in \U{\alpha}{\bm{\omega}^0}{{\bf c}^0}{L},$ and
\begin{align*}
  u\sts{t}\in \U{\alpha}{\bm{\omega}^0}{{\bf c}^0}{ \frac{L}{2} },
  &
  \text{ for any }  t\in\left(0,T^{\ast}\right],
\end{align*}
where $ \alpha <\alpha_{\ift} $ and $L>2L_{\ift},$
then there exist unqiue $\mathcal{C}^1$ functions
\begin{align*}
  \q\sts{t}:=\sts{\bm{\omega}\sts{t},\mathbf{c}\sts{t},\mathbf{x}\sts{t},\bm{\gamma}\sts{t} }
  :
  \left[0,T^{\ast}\right]\mapsto \R^8
\end{align*}
with $ -2\sqrt{\omega_j(t)}<c_j(t)<2z_0\sqrt{\omega_j(t)}$ for all $t\in [0, T^*], j=1,2 ,$ such that
\begin{align}
\label{orth:dyn}
  \inprod{\eps(t)}{ R_j(t) }
  =
  \inprod{\eps(t)}{\i\partial_x{R_j}(t) }
  =
  \inprod{\eps(t)}{\i R_j\sts{t}}=\inprod{ \eps (t)}{ \partial_x{R_j}\sts{t} }
  =0,
\end{align}
where $R_j\sts{t,x}=Q_{\omega_j\sts{t},c_j\sts{t}}\sts{x-x_j\sts{t}}\e^{\i\gamma_j\sts{t}}$, $j=1,2$, and
\begin{align}\label{decom1}
  \eps\sts{t,x}=u\sts{t,x}-\sum_{j=1}^{2} R_j\sts{t,x}.
\end{align}
Moreover, for $t\in [0, T^*]$, we have
\begin{gather}
\label{dynamic:est:ift}
  \norm{\eps\sts{t}}_{H^1}+\sum^{2}_{j=1}\left(\abs{ \omega_j\sts{t}-\omega_j^0 } + \abs{ c_j\sts{t}-c_j^0 }\right)
  <C_{\ift}\alpha,
\\
\label{dynamic:tech:nondeg}
    \frac{1}{2}<\frac{ \sqrt{ 4\omega_j\sts{t}-\sts{c_j\sts{t}}^2 } }{ \sqrt{ 4\omega_j^0-\sts{c_j^0}^2 } }<2,
\\
\label{dynamic:para:contr}
    \abs{ \dot{\omega}_k\sts{t} }
    +
    \abs{ \dot{c}_k\sts{t} }+
    \abs{ \dot{x}_k\sts{t}-c_k\sts{t} }
    +
    \abs{ \dot{\gamma}_k\sts{t}-\omega_k\sts{t} }
    \leq
    C_{\mathrm{abs}}\sts{\norm{\eps\sts{t}}_{H^1} +\e^{ -\theta_2 \sts{ L+\theta_2 t } } },
\\
\label{dist:center}
    x_2\sts{t}-x_1\sts{t}>\frac{1}{2}\sts{ L + \theta_2 t },
\end{gather}
where $ \theta_2=\min\ltl{ ~\frac{ \sqrt{4\omega_1^0-\sts{c_1^0}^2} }{8}~,~\frac{ \sqrt{4\omega_2^0-\sts{c_2^0}^2} }{8}~ , c_2^0-c_1^0~}. $
\end{lemma}
\begin{proof}
First, since $ u\sts{t}\in \U{\alpha}{\bm{\omega}^0}{{\bf c}^0}{ \frac{L}{2} } $  for any $t\in\left(0, T^{\ast}\right]$, there exist $\mathbf{ x }^0\sts{t}$
and
${\bm\gamma}^0\sts{t}$ such that
\begin{align}
\label{eq:1}
  \norm{
    u\sts{t}-\sum_{j=1}^{2} Q_{\omega_j^0,c_j^0}\sts{\cdot-x_j^0\sts{t}}\e^{\i\gamma_j^0\sts{t}}
  }_{H^1}<\alpha, \text{with }  x_2^0\sts{t}- x_1^0\sts{t}\geq\frac{L}{2} .
\end{align}
By Lemma \ref{lem:mod1},  we have the decomposition \eqref{orth:dyn} with
the estimates \eqref{dynamic:est:ift} and \eqref{dynamic:tech:nondeg}. Moreover, by the proof of Lemma \ref{lem:mod1},
we can obtain the estimate on $\mathbf{x(t)}$, i.e.
$$ \abs{ x_j(t)-x_j^0(t) }<C_{\ift}\alpha, $$
which together with  $x_2^0(t) - x_1^0(t) \geq L/2$ implies that
\begin{align}
\label{eq:3}
  x_2(t)-x_1(t)
  >\frac{L}{4}.
\end{align}
for sufficiently small $\alpha$ and sufficiently large $L$.

Now, we turn to the proof of \eqref{dynamic:para:contr}.
The rigorous calculations for \eqref{dynamic:para:contr} can be obtained by
Lemma 4 in \cite{MartelM:Instab:gKdV}. Here, we only give the formally calculations.
On one hand, by the equation \eqref{gDNLS} and the decomposition \eqref{decom1}, we have
\begin{align}
\label{eq:eps}
\notag
    0=&\i \partial_t\eps +\partial_{xx}\eps
    - \sum_{k=1}^{2}\i\sts{ \dot{x}_k-c_k }\partial_xR_k
    - \sum_{k=1}^{2}\sts{ \dot{\gamma}_k-\omega_k }R_k
    + \sum_{k=1}^{2}\i \dot{\omega}_k\partial_{\omega_k}R_k
    \\
\notag
    &
    + \sum_{k=1}^{2}\i \dot{c}_k\partial_{c_k}R_k
    + \i\abs{\sum_{k=1}^{2} R_k + \eps }^{2\sigma}\partial_x\sts{\sum_{k=1}^{2} R_k + \eps}
    - \sum_{k=1}^{2}\i\abs{R_k}^{2\sigma}\partial_xR_k
    \\
\notag
    =&\i \partial_t\eps +\partial_{xx}\eps
    - \sum_{k=1}^{2}\i\sts{ \dot{x}_k-c_k }\partial_xR_k
    - \sum_{k=1}^{2}\sts{ \dot{\gamma}_k-\omega_k }R_k
    + \sum_{k=1}^{2}\i \dot{\omega}_k\partial_{\omega_k}R_k
    \\
    &
    + \sum_{k=1}^{2}\i \dot{c}_k\partial_{c_k}R_k
    + \bigo{
          \abs{{\mathcal R}_1{\mathcal R}_2}
        + \abs{ \eps } + \abs{\partial_x\eps}
        },
  \end{align}
where we used
\begin{align*}
    \i \partial_t{R}_k + \partial_{xx}R_k
    =
    &
    -\i\abs{R_k}^{2\sigma}\partial_xR_k - \omega_k R_k + \i\dot{\omega}_k\partial_{\omega_k}R_k
    \\
    &
    + \i\dot{c}_k \partial_{c_k}{R}_k
    -\i\sts{ \dot{x}_k-c_k }\partial_x{R}_k
    -\sts{ \dot{\gamma}_k -\omega_k } R_k,
\end{align*}
and $ \abs{ \eps }+\abs{ \mathcal{ R }_1 }+\abs{ \mathcal{ R }_2 }\lesssim 1 $
with $\mathcal{R}_k$ is one of $ R_k $ and $\partial_xR_k$,
Then, by \eqref{eq:eps} and the orthogonal condition \eqref{orth:dyn}, we have
\begin{align}
\label{eq:4}
    \abs{ \dot \omega_k (t)} + \abs{ \dot c_k(t) } + \abs{ \dot{x}_k(t)-c_k(t) } + \abs{ \dot{\gamma}_k(t)-\omega_k (t)}
    \leq
     C_{ \mathrm{abs} }\left(\norm{\eps}_{H^1} + \e^{ -2\theta_2\abs{x_1(t)-x_2(t)} }\right),
\end{align}
where we used the fact:
  \begin{align}
  \label{38}
    \int \abs{\mathcal{R}_1\mathcal{R}_2}
    \leq C_{\mathrm{abs}}
    \int    \e^{ -\frac{ \sqrt{ 4\omega_1(t)-\sts{c_1(t)}^2 } }{2}\abs{x-x_1(t)} }
            \e^{ -\frac{ \sqrt{ 4\omega_2(t)-\sts{c_2(t)}^2 } }{2}\abs{x-x_2(t)} }\dx
    \leq C_{\mathrm{abs}}\e^{ -2\theta_2\abs{x_1(t)-x_2(t)} }.
  \end{align}
Inserting \eqref{eq:3} into \eqref{eq:4},
we obtain the following "rough" estimate
\begin{align}
\label{eq:5}
   \abs{ \dot \omega_k (t)} + \abs{ \dot c_k(t) } + \abs{ \dot{x}_k(t)-c_k(t) } + \abs{ \dot{\gamma}_k(t)-\omega_k (t)}
    \leq
     C_{ \mathrm{abs} }\left(\norm{\eps}_{H^1} + \e^{ -\frac{\theta_2}{2}L}\right).
\end{align}
On the other hand, combining \eqref{dynamic:est:ift}
with \eqref{eq:5}, we have
  \begin{align}
  \label{eq:6}
  \notag
    \dot{x}_2(t)-\dot{x}_1(t)
    =& \sts{ \dot{x}_2(t) - c_2(t) } - \sts{ \dot{x}_1(t) - c_1(t) } + \sts{ c_2(t)-c_2^0 } - \sts{ c_1(t)-c_1^0 } + \sts{ c_2^0-c_1^0 }
    \\
    \notag
    \geq & \sts{ c_2^0-c_1^0 } - C_{\mathrm{abs}}\alpha - C_{\mathrm{abs}} \e^{-\frac{\theta_2}{4} L}
    \\
    \geq & \frac{1}{2}\sts{ c_2^0-c_1^0 },
  \end{align}
then integrating \eqref{eq:6}, we obtain
  \begin{align*}
    x_2(t)-x_1(t)
    \geq
    x_2(0)-x_1(0) + \frac{1}{2}\int_{0}^{t}\sts{ c_2^0-c_1^0 }\d s
    \geq \frac{L}{2}+\frac{1}{2}\sts{ c_2^0-c_1^0 }t,
  \end{align*}
which implies that
  \begin{align*}
  \abs{ \dot \omega_k (t)} + \abs{ \dot c_k(t) } + \abs{ \dot{x}_k(t)-c_k(t) } + \abs{ \dot{\gamma}_k(t)-\omega_k (t)}
    \leq
    C_{ \mathrm{abs} }\left( \norm{\eps}_{H^1} + \e^{ -\theta_2\sts{ L+\theta_2 t } }\right).
  \end{align*}
This concludes the proof.
\end{proof}

\section{Monotonicity formula}\label{sect:mon}
\label{sec:mono}

In \cite{MiaoTX:DNLS:Stab}, under the non-degenerate condition
\begin{align*}
  \det\ntn{ d''\sts{\omega, c} } <0,
\end{align*}Miao, Tang and Xu obtained the orbital stability of the single solitary wave of the equation \eqref{gDNLS} with $\sigma=1$ in $H^1(\R)$ by the conservation laws of the energy, mass and momentum, these conservation laws were used to refine the estimates about the radiation term $\eps(t)$ and parameters variance $\abs{\omega(t)-\omega(0)} + \abs{c(t)-c(0)}$.
In this section, because of the multi-dimension of parameters $ \bm{\omega}, \mathbf{c}$ in dealing with the multi-solitary waves, we will introduce the analogue monotonicity formulas as those in \cite{LeWu:DNLS},\cite{MiaoTX:DNLS:Stab} instead of the conservation laws to refine the estimates \eqref{dynamic:est:ift}
and \eqref{dynamic:para:contr} about the radiation term $\eps(t)$ and parameters variance $\abs{\omega_k(t)-\omega_k(0)}$ and $\abs{c_k(t)-c_k(0)}$, $k=1, 2$. Those monotonicity formulas are  related to the localized mass and momentum.

We first give a Virial type identity.
\begin{lemma}\label{lem:virial}
Let $g:\R\mapsto\R$ be a $\mathcal{C}^3$ real-valued function such that
$g'$, $g''$ and $g'''$ are bounded.
If $u\in \mathcal{C}\sts{ [0, T^{\ast}], H^1 }$ is a solution of \eqref{gDNLS},
then, for all $t\in[0, T^*]$, we have
\begin{align*}
      \frac{\d}{\d t}\int \abs{u}^2~g
        =&
        2\Im\int\bar{u}u_x~g' +\frac{1}{\sigma}\int\abs{u}^{2\sigma+2}~g'.
        \\
        -\frac{\d}{\d t}\Im\int \bar{u}u_x~g
        =
        &
        -2\int\abs{u_x}^2~g'
        - \Im\int\abs{u}^{2\sigma}\bar{u}u_x~g'
        +\frac{1}{2}\int\abs{u}^2~g'''.
\end{align*}
\end{lemma}
\begin{proof}
It follows from simple computations.
\end{proof}

Now suppose $\triangle\omega\in\R,$ $\triangle c\in\R,$ $ \bar{x}^0\in\R, $ $\mu \in\R,$ and $a>0$, then by Lemma \ref{lem:virial}, we have for any $t\in [0, T^*]$,
    \begin{align}
    \label{viri:1}
    \notag
      &\frac{\d}{\d t}
      \ltl{
        \frac{\triangle \omega}{2}\int\abs{u}^2~g\sts{ \frac{x-\bar{x}^0 -\mu t }{ \sqrt{t+a} } }
        -
        \frac{\triangle c}{2}\Im\int\bar{u}u_x~g\sts{ \frac{x-\bar{x}^0 -\mu t }{ \sqrt{t+a} } }
       }
      \\
      \notag
      =&
        \frac{1}{\sqrt{t+a}}
            \int
                \ltl{
                    - \triangle c\abs{u_x}^2
                    + \sts{ \triangle\omega + \frac{ \mu \triangle c }{ 2 } }\Im\sts{\bar{u}u_x}
                    - \frac{\mu\triangle\omega}{2}\abs{u}^2
                    }~g'\sts{ \frac{x-\bar{x}^0 -\mu t }{ \sqrt{t+a} } }
      \\
      \notag
      & + \frac{ 1 }{ \sts{t+a} }
        \int
            \ltl{
                \frac{ \triangle c }{ 4 }\Im\sts{ \bar{u}u_x } - \frac{\triangle\omega}{4}\abs{u}^2
            }\Lambda g\sts{ \frac{x-\bar{x}^0 -\mu t }{ \sqrt{t+a} } }
        + \frac{ \triangle c }{ 4 (t+a)^{3/2} }\int\abs{u}^2~g'''\sts{ \frac{x-\bar{x}^0 -\mu t }{ \sqrt{t+a} } }
      \\
      &
        - \frac{\triangle c}{ 2 \sqrt{t+a} }
            \Im\int\abs{u}^{2\sigma}\bar{u}u_x~g'\sts{ \frac{x-\bar{x}^0 -\mu t }{ \sqrt{t+a} } }
        + \frac{ \triangle\omega }{ 2\sts{\sigma+1}\sqrt{t+a} }
            \int\abs{u}^{2\sigma+2}~g'\sts{ \frac{x-\bar{x}^0 -\mu t }{ \sqrt{t+a} } },
    \end{align}
where $ \sts{ \Lambda g }\sts{x}:=xg'\sts{x} .$

Now let
 \begin{align*}
 \bar{x}^0 = \frac{x^0_1 + x^0_2}{2},\quad \mu = 2\frac{\omega_2\sts{0}-\omega_1\sts{0}}{c_2\sts{0}-c_1\sts{0}},
 \end{align*}
 and define the following functional
\begin{align}
\label{j:sum}
\notag
  \mathcal{J}_{\mathrm{sum}}\sts{t}
  =
  &
            \frac{ \omega_1\sts{0} }{2}\int\abs{u(t,x)}^2~\varphi\sts{ - \frac{ x-\bar{x}^0-\mu t}{ \sqrt{t+a} } }
         -
         \frac{ c_1\sts{0} }{2}\Im\int\bar{u}u_x~\varphi\sts{ - \frac{ x-\bar{x}^0-\mu t}{ \sqrt{t+a} }}
   \\
  &
    +
         \frac{ \omega_2\sts{0} }{2}\int\abs{u(t,x)}^2~\varphi\sts{ \frac{ x-\bar{x}^0-\mu t}{ \sqrt{t+a} } }
         -
         \frac{ c_2\sts{0} }{2}\Im\int\bar{u}u_x~\varphi\sts{ \frac{ x-\bar{x}^0-\mu t}{ \sqrt{t+a} } },
\end{align}
which is used to capture the localized mass and momentum around each solitary waves.  According to the weak interactions between the solitary waves, we have the following monotonicity properties.

\subsection{Monotone result for the line $\bar{x}^0 +\mu t$}
\begin{proposition}
\label{pro:mono1}
  Let $a=L^2/64$ and $u\in\mathcal{C}\sts{ \left[0, T^{\ast}\right], H^1 }$
  be a solution satisfying the assumption of Lemma \ref{lem:mod2}.
  Then, there exists $C_{\mathrm{abs}}$ such that
  \begin{align*}
    \frac{\d}{\d t}\mathcal{J}_{\mathrm{sum}}\sts{t}
    \leq
         \frac{C_{\mathrm{abs}}}{(t+a)^{3/2}}\norm{\eps\sts{t}}_{H^1}^2 +   \frac{C_{\mathrm{abs}}}{(t+a)^{3/2}}\e^{-\theta_3\sts{ L+\theta_3 t }},
  \end{align*}
  where
    $
      \theta_3=\min\ltl{ ~\frac{ \sqrt{4\omega_1^0-\sts{c_1^0}^2} }{64}~,~\frac{ \sqrt{4\omega_2^0-\sts{c_2^0}^2} }{64}~, 4\sts{ \mu -c_1^0}~,~4\sts{c_2^0- \mu }~}.
    $
\end{proposition}
We rewrite $\mathcal{J}_{\mathrm{sum}}\sts{t}$
as the following identity
\begin{align}
\label{j2}
  \mathcal{J}_{\mathrm{sum}}\sts{t}
  =
  &
        \frac{ \omega_1\sts{0} }{2}\int\abs{u(t,x)}^2
         -
         \frac{ c_1\sts{0} }{2}\Im\int\sts{\bar{u}u_x}\sts{t,x}
    +
    \mathcal{J}\sts{t}
\end{align}
where
\begin{align*} \mathcal{J}\sts{t} =  & \frac{\omega_2\sts{0}-\omega_1\sts{0}}{2}\int\abs{u(t,x)}^2~\varphi\sts{ \frac{x-\bar{x}^0 -\mu t }{ \sqrt{t+a} } } \\
 & - \frac{c_2\sts{0}-c_1\sts{0}}{2}\Im\int\sts{\bar{u}u_x}\sts{t,x}~\varphi\sts{ \frac{x-\bar{x}^0 -\mu t }{ \sqrt{t+a} } } .
 \end{align*}
 By the conservation of mass and momentum, it suffices to show
\begin{proposition}\label{prop:mono1}
  Let $a$, $\theta_3$ and $u$ be as those in Proposition \ref{pro:mono1}.
  Then, there exists $C_{\mathrm{abs}}$ such that
  \begin{align*}
    \frac{\d}{\d t}\mathcal{J}\sts{t}
    \leq \frac{C_{\mathrm{abs}}}{(t+a)^{3/2}}\norm{\eps\sts{t}}_{H^1}^2 +
            \frac{C_{\mathrm{abs}}}{(t+a)^{3/2}}\e^{-\theta_3\sts{ L+\theta_3 t }}.
  \end{align*}
Moreover, we have
\begin{align}
\label{mono1}
  \mathcal{J}\sts{t} - \mathcal{J}\sts{0}
    \leqslant
    \frac{C_{\mathrm{abs}}}{ L }\sup_{0<s<t}\norm{\eps\sts{s}}_{H^1}^2
    +
    C_{\mathrm{abs}}\e^{-\theta_3L}.
\end{align}
\end{proposition}
Before the proof of Proposition \ref{prop:mono1}, we fist give the following estimate.
\begin{lemma}
\label{lem:p est}
    Let $\Omega_{w}:=\set{x\in\R }{\abs{ x-\bar{x}^0-\mu t }<\sqrt{t+a}}, $ then for any $2\leq p<\infty$, we have
    \begin{align}
    \label{esti:wkdom}
      \int_{ \Omega_{w} }\abs{u}^p
      \leq
      C_{\mathrm{abs}}\e^{-\theta_3\sts{ L+\theta_3 t }}+ C_{\mathrm{abs}}\norm{\eps\sts{t}}_{H^1}^p
      \leq \frac{1}{32},
    \end{align}
    where
    $
      \theta_3
    $ is given by Proposition \ref{pro:mono1}.
\end{lemma}
\begin{proof}
By Lemma \ref{lem:mod2}, the solution $u\sts{t}$ can be decomposed as
 \begin{align*}
  u\sts{t}= \sum_{j=1}^{2} Q_j\sts{\cdot-x_j\sts{t}}\e^{\i\gamma_j\sts{t}} + \eps\sts{t},
\end{align*}
and \eqref{orth:dyn}-\eqref{dynamic:para:contr} hold. Then we have
\begin{align}
\label{9}
  \int_{ \Omega_{w} }\abs{u}^p
  \leq
  \sum_{j=1}^{2} C_{\mathrm{abs}}\int_{ \Omega_{w} }\abs{Q_j\sts{\cdot-x_j\sts{t}}}^p
  +
  C_{\mathrm{abs}}\int_{ \Omega_{w} }\abs{ \eps }^p.
\end{align}
We first estimate the contribution from $Q_1$. If $x\in \Omega_{w} $, we obtain $\abs{x-\overline{x}^{0}-\mu
t}<\sqrt{t+a}\leqslant \sqrt{t} + \frac{L}{8},$ and
\begin{align}
\label{eq:tw1:dist} \notag
    \abs{ x-x_{1}\sts{t} }
&
=
    \abs{ \sts{ x-\overline{x}^{0}-\mu t } -\sts{ x_{1}\sts{t} -\overline{x}^{0}-\mu t } }
\\
\notag
&
\geqslant
    \abs{ x_{1}\sts{t} -\overline{x}^{0}-\mu t } - \abs{ x-\overline{x}^{0}-\mu t }
\\
&
\geqslant
    \abs{ x_{1}\sts{t} -\overline{x}^{0}-\mu t } - \sqrt{t} - \frac{L}{8}.
\end{align}
By \eqref{dynamic:est:ift} and \eqref{dynamic:para:contr}, we have for
sufficiently small $\alpha$ and sufficiently large $L$ that
\begin{align*}
    \frac{d}{dt}\sts{\overline{x}^{0}+\mu t - x_{1}\sts{t}}
& =
    \mu - \sts{ \dot{x}_{1}\sts{t} - c_{1}\sts{t} } - c_{1}\sts{t}
\\
&
\geqslant
    \mu - C_{\mathrm{abs}}\left(C_{\ift}\alpha +
                e^{ -\theta_2 \sts{ L + \theta_2 t } } \right) - c_{1}^0 - C_{\ift}~\alpha
\\
&
\geqslant
    \frac{\mu - c_{1}^0}{2},
\end{align*}
and so,
\begin{align*}
    \overline{x}^{0}+\mu t - x_{1}\sts{t}
 \geqslant
    \overline{x}^{0} - x_{1}\sts{0} + \frac{ \mu - c_{1}^0}{2}\;t
\geqslant
    \frac{L}{4} + \frac{\mu - c_{1}^0}{2}\;t.
\end{align*}
Now inserting the above estimate into \eqref{eq:tw1:dist},
 we obtain for sufficiently small $\alpha$ and sufficiently large $L$ that
\begin{align*}
    \abs{ x-x_{1}\sts{t} }
 & \geqslant
     \overline{x}^{0}  +\mu t - x_{1}\sts{t}   - \sqrt{t} -
     \frac{L}{8}
\\
&
\geqslant
       \frac{L}{16} + \frac{\mu - c_{1}^0}{4}\;t
    +
    \sts{ \frac{\mu - c_{1}^0}{4}\;t - \sqrt{t} + \frac{L}{16}}
\\
&
\geqslant
    \frac{L}{16} + \frac{\mu - c_{1}^0}{4}\;t.
\end{align*}
Then, it follows from the explicit expression of $Q_{\omega,c}$ that
\begin{align}
\label{7}
  \int_{\Omega_w}\abs{ Q_1\sts{ \cdot-x_1\sts{t} } }^p
  \leq C_{\mathrm{abs}} \e^{-\theta_3\sts{ L+\theta_3 t }},
\end{align}
where we used the fact that $p\geq 2$. By the similar argument, we have
\begin{align}
\label{8}
  \int_{\Omega_w}\abs{ Q_2\sts{ \cdot-x_2\sts{t} } }^p
  \leq C_{\mathrm{abs}} \e^{-\theta_3\sts{ L+\theta_3 t }}.
\end{align}
By \eqref{7} and \eqref{8} and the Sobolev inequality, we have
\begin{align*}
   \int_{\Omega_w}\abs{u}^p
    \leq
    &
    C_{\mathrm{abs}}\e^{-\theta_3\sts{ L+\theta_3 t }} + C_{\mathrm{abs}}\norm{\eps\sts{t}}_{H^1}^p.
\end{align*}
This concludes the proof.
\end{proof}
Now, let us prove Proposition \ref{prop:mono1}.
\begin{proof}[Proof of Proposition \ref{prop:mono1}]
First of all, let
$$
  v\sts{t,x}:=u\sts{t,x}\e^{-\i\frac{1 }{ 2 }\mu x},
$$ then we have
\begin{gather*}
    -\abs{u_x}^2 +   \mu \Im\sts{\bar{u}u_x} - \frac{\mu^2}{4}\abs{u}^2
  =
    \frac{1}{4}\Im\sts{\bar{v}v_x},
\\
    -\frac{1}{2}\Im\sts{ \abs{u}^{2\sigma}\bar{u}u_x }
  =
   -\frac{1}{2}\abs{v}^{2\sigma}\Im\sts{ \bar{v}v_x } -\frac{ \mu }{4}\abs{ v }^{2\sigma+2 }.
   \end{gather*}
Simple calculations yield that
\begin{align}
\label{17}
 & \frac{1}{c_2(0)-c_1(0)}\frac{\d}{\d t}  \mathcal{J}\sts{t}
  \\
  = &  -\frac{1}{\sqrt{t+a}}\int\abs{v_x}^2\varphi'\sts{\frac{ x-\bar{x}^0-\mu t}{ \sqrt{t+a} }}
        +\frac{1}{ 4(t+a)^{3/2} }\int\abs{v}^2~\varphi'''\sts{\frac{ x-\bar{x}^0-\mu t}{ \sqrt{t+a} }} \notag
  \\
  \label{1}
  &
        + \frac{1}{4\sts{t+a}}\Im\int\bar{v}v_x \Lambda\varphi\sts{\frac{ x-\bar{x}^0-\mu t}{ \sqrt{t+a} }}
  \\
  \label{2}
  &
        -\frac{1}{2\sqrt{t+a} }\Im\int \abs{v}^{2\sigma}\bar{v}v_x~\varphi'\sts{\frac{ x-\bar{x}^0-\mu t}{ \sqrt{t+a} }}
  \\
  \label{3}
  &
        -\frac{\sigma \mu }{4\sqrt{t+a}}\int\abs{v}^{2\sigma+2}~\varphi'\sts{\frac{ x-\bar{x}^0-\mu t}{ \sqrt{t+a} }}.
\end{align}

Next, we estimate \eqref{1}-\eqref{3} separately.

  {\textbf{\textit{Estimate for \eqref{1}.}}} The definition of $ \Lambda\varphi $ immediately implies that
  $ \abs{ \Lambda\varphi }\leq\varphi' ,$ which together with the Cauchy-Schwarz inequality, yields that
    \begin{align}
    \label{16}
    \notag
      & \abs{ \frac{1}{4\sts{t+a}}\Im\int\bar{v}v_x \Lambda\varphi\sts{\frac{ x-\bar{x}^0-\mu t}{ \sqrt{t+a} }} }\\
            \leq
      &
       \frac{1}{4\sts{t+a}} \sqrt{ \int\abs{v_x}^2~\varphi'\sts{\frac{ x-\bar{x}^0-\mu t}{ \sqrt{t+a} }}\cdot \int\abs{v}^2~\varphi'\sts{\frac{ x-\bar{x}^0-\mu t}{ \sqrt{t+a} }}  }
     \notag \\
      \leq
      &
        \frac{1}{4\sqrt{t+a}} \int\abs{v_x}^2~\varphi'\sts{\frac{ x-\bar{x}^0-\mu t}{ \sqrt{t+a} }}
        +
        \frac{1}{4(t+a)^{3/2}} \int\abs{v}^2~\varphi'\sts{\frac{ x-\bar{x}^0-\mu t}{ \sqrt{t+a} }}
    \end{align}

    {\textbf{\textit{Estimate for \eqref{2}.}}}
    Applying the Cauchy-Schwarz inequality, we have
    \begin{align}
 \notag
      & \abs{ \frac{1}{2\sqrt{t+a} }\Im\int \abs{v}^{2\sigma}\bar{v}v_x~\varphi'\sts{\frac{ x-\bar{x}^0-\mu t}{ \sqrt{t+a} }} } \\
      \leq
      &
      \frac{1}{2\sqrt{t+a} }
        \sqrt{ \int \abs{v}^{4\sigma+2}~\varphi'\sts{\frac{ x-\bar{x}^0-\mu t}{ \sqrt{t+a} }}\cdot \int \abs{v_x}^2 ~\varphi'\sts{\frac{ x-\bar{x}^0-\mu t}{ \sqrt{t+a} }} }
     \notag \\
      \leq
      &
      \frac{1}{4\sqrt{t+a}}\int \abs{v_x}^2 ~\varphi'\sts{\frac{ x-\bar{x}^0-\mu t}{ \sqrt{t+a} }}
      +
      \frac{1}{\sqrt{t+a}}\int \abs{v}^{4\sigma+2}~\varphi'\sts{\frac{ x-\bar{x}^0-\mu t}{ \sqrt{t+a} }}.    \label{13}
    \end{align}
    By the H\"{o}lder inequality, we have
    \begin{align*}
          \int\abs{v}^{4\sigma+2}~\varphi'\sts{\frac{ x-\bar{x}^0-\mu t}{ \sqrt{t+a} }}
                 \leq
                 \left\| \abs{v}^2\sqrt{ \varphi'\sts{\frac{ x-\bar{x}^0-\mu t}{ \sqrt{t+a} }} } \right\|_{L^\infty}^2
            \int_{\supp{\varphi'}}\abs{v}^{4\sigma-2}.
    \end{align*}
    By the  Sobolev inequality in Lemma 5.2 in \cite{MiaoTX:DNLS:Stab} and Lemma \ref{lem:p est}, we have
    \begin{align*}
          \left\| \abs{v}^2\sqrt{ \varphi'(\frac{ x-\bar{x}^0-\mu t}{ \sqrt{t+a} }) } \right\|_{L^\infty}^2
          \leq
          &
            8\int \abs{v_x}^2~\varphi'\int_{\supp{\varphi'}}\abs{u}^2
            +
            \frac{1}{2(t+a)}
            \int \abs{v}^2\frac{ \varphi''^2 }{ \varphi' }\int_{\supp{\varphi'}}\abs{v}^2
          \\
          \leq
          &
              8\int_{\supp{\varphi'}}\abs{v}^2    \int \abs{v_x}^2~\varphi'
            +
              \frac{1}{2(t+a)}\int_{\supp{\varphi'}}\abs{v}^2
            \int \abs{v}^2\frac{ \varphi''^2 }{ \varphi' }
            \\
           \leq &   8\int_{\supp{\varphi'}}\abs{v}^2    \int \abs{v_x}^2~\varphi'
            +\frac{1}{2(t+a)}
            \sts{\int_{\supp{\varphi'}}\abs{v}^2 }^2
            \\
             \leq & \frac14   \int \abs{v_x}^2~\varphi'
            + \frac{C_{\mathrm{abs}}}{t+a},
    \end{align*}
    which implies that
    \begin{multline*}
        \frac{1}{\sqrt{t+a}} \int\abs{v}^{4\sigma+2}~\varphi'\sts{\frac{ x-\bar{x}^0-\mu t}{ \sqrt{t+a} }}
           \\
           \leq
          \frac{1}{4\sqrt{t+a}}\int \abs{v_x}^2~\varphi'\sts{\frac{ x-\bar{x}^0-\mu t}{ \sqrt{t+a} }} + \frac{C_{\mathrm{abs}}}{ (t+a)^{3/2} }
          \sts{ \e^{-\theta_3\sts{ L+\theta_3 t }}+ \norm{\eps\sts{t}}_{H^1}^{4\sigma-2} }
    \end{multline*}
    Now inserting the above estimate into \eqref{13}, we have
    \begin{multline}
    \label{15}
      \abs{ \frac{1}{2\sqrt{t+a} }\Im\int \abs{v}^{2\sigma}\bar{v}v_x~\varphi'\sts{\frac{ x-\bar{x}^0-\mu t}{ \sqrt{t+a} }} }
      \\
      \leq
      \frac{1}{2\sqrt{t+a}}\int \abs{v_x}^2 ~\varphi'\sts{\frac{ x-\bar{x}^0-\mu t}{ \sqrt{t+a} }}
      +
      \frac{C_{\mathrm{abs}}}{ (t+a)^{3/2} }
      \sts{ \e^{-\theta_3\sts{ L+\theta_3 t }}+ \norm{\eps\sts{t}}_{H^1}^{4\sigma-2} }.
    \end{multline}

    {\textbf{\textit{Estimate for \eqref{3}.}}} By $\mu >0$, $\varphi'>0$, we have
    \begin{align}
    \label{18}
      -\frac{\sigma\mu}{4\sqrt{t+a}}\int\abs{v}^{2\sigma+2}~\varphi'\sts{\frac{ x-\bar{x}^0-\mu t}{ \sqrt{t+a} }}\leq 0.
    \end{align}

Inserting \eqref{16}, \eqref{15} and \eqref{18} into \eqref{17}, we can obtain the result.
\end{proof}

\subsection{Monotone result for different lines}

Now let $$ \theta_4:=\min\ltl{~\frac{ \sqrt{4\omega_1^0-\sts{c_1^0}^2} }{64},~\frac{ \sqrt{4\omega_2^0-\sts{c_2^0}^2} }{64},~\mu-2c_1^0,~4c_2^0-2\mu~}\leq \theta_3,$$
\begin{align*}
  \mu_{+,0}= \mu_{0,-}=4\frac{\omega_2(0)-\omega_1(0)}{c_2(0)-c_1(0)},\quad
  \mu_{-,0}=\mu_{0,+}=\frac{\omega_2(0)-\omega_1(0)}{c_2(0)-c_1(0)},
\end{align*}

\begin{align*}
    \phi_{\pm,0}\sts{t,x}:=\varphi\sts{ \frac{x-\bar{x}^0 -\mu_{\pm,0} t }{ \sqrt{t+a} } },
    & \quad
    \phi_{0,\pm}\sts{t,x}:=\varphi\sts{ \frac{x-\bar{x}^0 -\mu_{0,\pm} t }{ \sqrt{t+a} } },
    \end{align*}
    and define
\begin{align*}
  \mathcal{J}_{+,0}(t) &= \left(\omega_2(0)-\omega_1(0)\right)\int\abs{u(t,x)}^2~\phi_{+,0}(t,x) - \frac{c_2(0)-c_1(0)}{2}\Im\int\bar{u}u_x~\phi_{+,0}(t,x),
  \\
  \mathcal{J}_{-,0}(t) &= \frac{\left(\omega_2(0)-\omega_1(0)\right)}{4}\int\abs{u(t,x)}^2~\phi_{-,0}(t,x) - \frac{c_2(0)-c_1(0)}{2}\Im\int\bar{u}u_x~\phi_{-,0}(t,x),
  \\
  \mathcal{J}_{0,+}(t) &= \frac{\left(\omega_2(0)-\omega_1(0)\right)}{2}\int\abs{u(t,x)}^2~\phi_{0,+}(t,x) - \left(c_2(0)-c_1(0)\right)\Im\int\bar{u}u_x~\phi_{0,+}(t,x),
  \\
  \mathcal{J}_{0,-}(t) &= \frac{\left(\omega_2(0)-\omega_1(0)\right)}{2}\int\abs{u(t,x)}^2~\phi_{0,-}(t,x) - \frac{c_2(0)-c_1(0)}{4}\Im\int\bar{u}u_x~\phi_{0,-}(t,x).
\end{align*}
By the analogue proof as that in Proposition \ref{prop:mono1}, we have
\begin{corollary}
\label{coro:mono2}
  Let $u\in\mathcal{C}\sts{ \left[0,T^*\right], H^1 }$
  be a solution satisfying the assumption of Lemma \ref{lem:mod2}.
  Then, there exists $C_{\mathrm{abs}}$ such that
  \begin{align*}
      \mathcal{J}_{\pm,0}\sts{t}-\mathcal{J}_{\pm,0}\sts{0}
        &
        \leqslant
        \frac{C_{\mathrm{abs}}}{ L }\sup_{0<s<t}\norm{\eps\sts{s}}_{H^1}^2
        +
        C_{\mathrm{abs}}\e^{-\theta_4L},
      \\
      \mathcal{J}_{0,\pm}\sts{t}-\mathcal{J}_{0,\pm}\sts{0}
        &
        \leqslant
        \frac{C_{\mathrm{abs}}}{ L }\sup_{0<s<t}\norm{\eps\sts{s}}_{H^1}^2
        +
        C_{\mathrm{abs}}\e^{-\theta_4L}.
  \end{align*}
\end{corollary}

\section{Proof of Theorem \ref{mainthm}}\label{sect:proof}
Let $\sigma\in\sts{1,2}$ and $z_0=z_0(\sigma)\in\sts{0,1}$ satisfy $ F(z_0;\sigma) =0,$
    where $F(z;\sigma)$ is defined by \eqref{eq:sig-z}. Let $\omega^{0}_k$ and $c^{0}_k$ satisfy the assumptions in Theorem
\ref{mainthm}. Let $\alpha_{0}$ be defined by Lemma \ref{lem:mod1}, and
$A_0>2$, $\delta_0=\delta_0(A_0), L_0=L_0(A_0)$ be chosen later.
Suppose that $u\sts{t}$ is the solution of \eqref{gDNLS} with initial
data
$
  u_0\in \U{\alpha}{\bm{\omega}^0}{{\bf c}^0}{L},
$
and define
\begin{align}
\label{def:Tast}
  T^{\ast}
  :=
  \sup
  \set{ t\geq 0 }
    {
               \sup_{\tau\in\left[0,t\right]}\inf_{\substack{x_{2}^0-x_{1}^0 > \frac{L}{2} \\ \gamma^0_1,~\gamma^0_2\in\R}}
        \left\| u(\tau,\cdot)-
            \sum_{j=1}^{2} Q_{\omega_j^0,c_j^0}\sts{\cdot-x_j^0}\e^{\i\gamma_j^0}
        \right\|_{H^1} \leq  A_0\sts{ \delta + \e^{-\frac{\theta_0}{2} L} }
    },
\end{align}
where
\begin{align}
\label{thta6}
  \theta_0=\min
  \ltl{
    \frac{ \sqrt{4\omega_1^0-\sts{c_1^0}^2} }{128},
    ~\frac{ \sqrt{4\omega_2^0-\sts{c_2^0}^2} }{128},
    ~c_2^0-c_1^0,
    ~2c_2^0- 2\mu,
    ~\mu-2c_1^0
  }.
\end{align}
By the continuity of $u(t)$ in $H^1$, we know that $T^*>0$. In order
to prove Theorem \ref{mainthm}, it suffices to show
$T^*=+\infty$ for some $A_0>2$, $\delta_0>0$, and $L_0$. We argue with contradiction. Suppose that $T^*<+\infty$, we know
that for any $t\in [0, T^*]$, there exist $(x^0_k(t),
\gamma^0_k(t))\in \R^2$, $k=1, 2$ such that $x^0_2(t)\geq
x^0_1(t)+\frac{L}{2}$ and
\begin{align*}
      \normhone{ u(t,\cdot)-\sum^2_{k=1}Q_{\omega^{0}_k,c^{0}_k}\sts{\cdot-x^{0}_k(t)}e^{\i\gamma^{0}_k(t)}}   \leq A_0\left(\delta +e^{-\theta_0
      \frac{L}{2}} \right).
\end{align*}

\begin{enumerate}[label=\emph{{Step \arabic*.}},ref=\emph{{Step \arabic*}}]

\item \textit{Decomposition of $u\sts{t}$.}
    Let $\widetilde{L}_{0}>0$ be
determined by Lemma \ref{lem:mod1}, and $L_2, L_3$ be
determined by Proposition \ref{prop:mono1} and Corollary
\ref{coro:mono2}, and choose $\delta_0>0$ small enough and $L_{0}$
large enough, such that for $\delta<\delta_0$ and
$L>L_0(A_0)>\max\ltl{ 2\widetilde{L}_{0}, L_2, L_3}$,
\begin{align*}
    A_{0}~\sts{ \delta + e^{ - \theta_{6} \frac{L}{2} }
    }<\alpha_{0}.
\end{align*}

By Lemma \ref{lem:mod1} and Lemma \ref{lem:mod2}, we have
    \begin{align}
    \label{19}
        u\sts{t, x}
        =
        \eps\sts{t, x}
        +
        \sum_{j=1}^{2}R_j\sts{t, x}
    \end{align}
  where
    $ R_j\sts{t, x} = Q_{ \omega_j\sts{t},c_j\sts{t} }\sts{ x-x_j\sts{t} }\e^{\i\gamma_j\sts{t} }$, and the orthogonality
    \begin{align}
    \label{53}
      \inprod{\eps(t)}{ R_{j} }
      =
      \inprod{\eps(t)}{\i\partial_x{ R_{j} } }
      =
      \inprod{\eps(t)}{ \i R_{j} }
      =
      \inprod{ \eps(t) }{ \partial_x{ R_{j}  } }=0,
    \end{align}
    hold for any $t\in\left[~0~,~T^{\ast}~\right]$. Moreover, we have
    \begin{gather}
    \label{41}
      \norm{\eps\sts{t}}_{H^1}+\sum^2_{j=1}\left(\abs{ \omega_j\sts{t}-\omega_j^0 } + \abs{ c_j\sts{t}-c_j^0 }\right)
      <C_{\ift}{ A_0\sts{ \delta + \e^{-\theta_0 \frac{L}{2}} }  },
      \\
      \label{42}
      \abs{ \dot{\omega}_k\sts{t} }
        +
        \abs{ \dot{c}_k\sts{t} }+
        \abs{ \dot{x}_k\sts{t}-c_k\sts{t} }
        +
        \abs{ \dot{\gamma}_k\sts{t}-\omega_k\sts{t} }
        \leq
        C_{\mathrm{abs}}\sts{\norm{\eps\sts{t}}_{H^1} +\e^{ -\theta_0 \sts{ L+\theta_0 t } } }
      \\
      \label{39}
      x_2\sts{t}-x_1\sts{t}>\frac{1}{2}\sts{ L + \theta_0 t }.
    \end{gather}
  In particular, we have
     \begin{gather}
      \norm{\eps(0)}_{H^1}+\sum^2_{j=1}\left(\abs{ \omega_j\sts{0}-\omega_j^0 } + \abs{ c_j\sts{0}-c_j^0 }\right)
      <C_{\ift}\delta.
    \end{gather}

    \step{\textit{ Refined estimate on $\norm{\eps}_{H^1}^2$ and $\abs{ \mathcal{J}\sts{t} - \mathcal{J}\sts{0} }$. }}In order to do so, we first introduce the functional
\begin{align*}
  \S\sts{t}:=E\sts{u\sts{t}} + \J_{\mathrm{sum}}\sts{u\sts{t}}.
\end{align*}
  and expand it as following
    \begin{lemma}
    \label{lem:tlexp}
    \begin{align}
    \label{eq:exp}
    \notag
         \S\sts{t}
         =
         &
         \sum_{k=1}^{2}S_{\omega_k\sts{0},c_k\sts{0}}\sts{ R_k\sts{0} }
         +
         \mathcal{H}\bilin{\eps\sts{t} }{\eps\sts{t}}
         \\
         \notag
         &
         +
        \bigo{  \sum_{k=1}^{2}\left(\abs{\omega_k\sts{t}-\omega_k\sts{0} }^2 + \abs{c_k\sts{t}-c_k\sts{0} }^2\right) }
         \\
         &
         +
         \norm{\eps\sts{t}}_{H^1}^2\smallo{ \norm{\eps\sts{t}}_{H^1} }
         +
         \bigo{ \e^{-\theta_0\sts{L + \theta_0 t }} }.
     \end{align}
     where
     \begin{align*}
       \mathcal{H}\bilin{\eps\sts{t} }{\eps\sts{t}}
       =
       &
       \frac{1}{2}\int\abs{\eps_x\sts{t}}^2
       +
       \frac{\omega_1\sts{t}}{2}\int \abs{\eps\sts{t}}^2~\sts{1-\phi}
       +
       \frac{\omega_2\sts{t}}{2}\int \abs{\eps\sts{t}}^2~\phi
       \\
       &
       +
       \frac{c_1\sts{t}}{2}\Im\int\bar{\eps}\eps_x~\sts{1-\phi}
       +
       \frac{c_2\sts{t}}{2}\Im\int\bar{\eps}\eps_x~\phi
       +
       \frac{1}{2}
       N\sts{\eps\sts{t}},
     \end{align*}
  and
     \begin{align*}
   N\sts{\eps}
  =
          &
          \sum_{k=1}^{2}
          \Im\int\abs{R_k}^{2\sigma}\bar{\eps} \eps_x
          +
          \sigma\sum_{k=1}^{2}
          \Im\int \abs{R_k}^{2\sigma-2}
          \sts{
            \bar{R}_k\partial_xR_k\abs{ \eps }^2 + R_k\partial_xR_k\bar{\eps}^2 }.
     \end{align*}
\end{lemma}
\begin{proof} Please refer to the proof in Appendix A.
\end{proof}
\begin{lemma}\label{lem:coer2}
  \begin{align*}
    \mathcal{H}\bilin{\eps }{\eps}\geq \kappa\norm{\eps}_{H^1}^2.
  \end{align*}
\end{lemma}

\begin{proof}
Please refer to Lemma 6.2 in \cite{LeWu:DNLS} and Lemma 6.2 in \cite{MiaoTX:DNLS:Stab}.
\end{proof}

    By Lemma \ref{lem:tlexp}, we have for all $t\in\left[0, T^{\ast}\right],$
        \begin{align}
        \label{20}
        \notag
            \S\sts{t}
            =
            &
            \sum_{k=1}^{2}S_{\omega_k\sts{0},c_k\sts{0}}\sts{ R_k\sts{0} }
            +
            \mathcal{H}\bilin{\eps\sts{t} }{\eps\sts{t}}
            \\
            \notag
            &
            +
           \bigo{  \sum_{k=1}^{2}\left(\abs{\omega_k\sts{t}-\omega_k\sts{0} }^2 + \abs{c_k\sts{t}-c_k\sts{0} }^2\right) }
            \\
            &
            +
            \norm{\eps\sts{t}}_{H^1}^2\smallo{ \norm{\eps\sts{t}}_{H^1} }
            +
            \bigo{ \e^{-\theta_0\sts{L + \theta_0 t }} }.
        \end{align}
   In particular, we have
        \begin{align*}
            \S\sts{0}
            =
            &
            \sum_{k=1}^{2}S_{\omega_k\sts{0},c_k\sts{0}}\sts{ R_k\sts{0} }
            +
            \mathcal{H}\bilin{\eps\sts{0}}{\eps\sts{0}}
            +
            \norm{\eps\sts{0}}_{H^1}^2\smallo{ \norm{\eps\sts{0}}_{H^1} }
            +
            \bigo{ \e^{-\theta_0 L } },
        \end{align*}
        which implies that
        \begin{align}
        \label{21}
          \sum_{k=1}^{2}S_{\omega_k\sts{0},c_k\sts{0}}\sts{ R_k\sts{0} }
            =
            &
            \S\sts{0}
            -
            \mathcal{H}\bilin{\eps\sts{0} }{\eps\sts{0}}
            +
            {\eps\sts{0}}_{H^1}^2\smallo{\norm{\eps\sts{0}}_{H^1}}
            +
            \bigo{ \e^{-\theta_0 L } }.
        \end{align}
        Inserting \eqref{21} into \eqref{20}, we obtain by  Lemma \ref{lem:coer2} and the conservation laws of mass, momentum and energy that
        \begin{align}
        \notag
          &
          \kappa\norm{\eps\sts{t}}_{H^1}^2
          \\
          \notag
          \leq
          &
          \mathcal{H}\bilin{\eps\sts{t} }{\eps\sts{t}}
          \\
          \notag
          =
          &
          \S\sts{t} - \sum_{k=1}^{2}S_{\omega_k\sts{0},c_k\sts{0}}\sts{ R_k\sts{0} }
          +
         \bigo{  \sum_{k=1}^{2}\left( \abs{\omega_k\sts{t}-\omega_k\sts{0} }^2 + \abs{c_k\sts{t}-c_k\sts{0} }^2\right) }
          \\
          \notag
          &
          +
          \norm{\eps\sts{t}}_{H^1}^2\smallo{ \norm{\eps\sts{t}}_{H^1} }
          +
          \bigo{ \e^{-\theta_0\sts{L + \theta_0 t }} }
          \\
          \notag
          =
          &
          \S\sts{t} - \S\sts{0}
          +
          \mathcal{H}\bilin{\eps\sts{0} }{\eps\sts{0}}
          +
          {\eps\sts{0}}_{H^1}^2\smallo{\norm{\eps\sts{0}}_{H^1}}
          +
          \bigo{ \e^{-\theta_0 L } }
          \\
          \notag
          &
          +
         \bigo{  \sum_{k=1}^{2}\left(\abs{\omega_k\sts{t}-\omega_k\sts{0} }^2 + \abs{c_k\sts{t}-c_k\sts{0} }^2 \right) }
          +
          \bigo{ \e^{-\theta_0\sts{L + \theta_0 t }} }
          \\
          \notag
          =
          &
          \J_{\mathrm{sum}}\sts{t} - \J_{\mathrm{sum}}\sts{0}
          +
          \mathcal{H}\bilin{\eps\sts{0} }{\eps\sts{0}}
          +
          {\eps\sts{0}}_{H^1}^2\smallo{\norm{\eps\sts{0}}_{H^1}}
          +
          \bigo{ \e^{-\theta_0 L } }
          \\
          &
          +
\bigo{ \sum_{k=1}^{2}\left(\abs{\omega_k\sts{t}-\omega_k\sts{0} }^2 + \abs{c_k\sts{t}-c_k\sts{0} }^2 \right)}
          +
          \bigo{ \e^{-\theta_0\sts{L + \theta_0 t }} }. \label{est:J abs}
        \end{align}
    By Proposition \ref{prop:mono1}, we obtain
        \begin{align}
          \norm{\eps\sts{t}}_{H^1}^2
          \leq
          &
          \frac{C_{\mathrm{abs}}}{ L }\sup_{0<s<t}\norm{\eps\sts{s}}_{H^1}^2
          +
          C_{\mathrm{abs}}\e^{-\theta_0 L }
          +
          C_{\mathrm{abs}}\norm{\eps\sts{0}}_{H^1}^2\notag
          \\
          &
          +
       C_{\mathrm{abs}}   \sum_{k=1}^{2}\left( \abs{\omega_k\sts{t}-\omega_k\sts{0} }^2 + \abs{c_k\sts{t}-c_k\sts{0} }^2 \right). \label{est:rem}
        \end{align}
Moreover, by \eqref{est:J abs}, we have
    \begin{align*}
      &
      \mathcal{J} \sts{ 0 }
        -
      \mathcal{J} \sts{ t }
      \\
      =
      &
    \mathcal{J}_{sum} \sts{ 0 }
        -
      \mathcal{J}_{sum} \sts{ t }
      \\
      =
      &
      \mathcal{H}\bilin{\eps\sts{0} }{\eps\sts{0}}
      -
      \mathcal{H}\bilin{\eps\sts{t} }{\eps\sts{t}}
      +
      \norm{\eps\sts{0}}_{H^1}^2\smallo{ \norm{\eps\sts{0}}_{H^1} }
      +
      \bigo{ \e^{-\theta_0 L } }
      \\
      &
      +
\bigo{ \sum_{k=1}^{2}\left(\abs{\omega_k\sts{t}-\omega_k\sts{0} }^2 + \abs{c_k\sts{t}-c_k\sts{0} }^2 \right)}
            +
            \bigo{ \e^{-\theta_0\sts{L + \theta_0 t }} }
      \\
      \leq
      &
        C_{\mathrm{abs}}\norm{\eps\sts{0}}_{H^1}^2
        +
        C_{\mathrm{abs}}\e^{-\theta_0 L }
      +
              C_{\mathrm{abs}}
      \sum_{k=1}^{2}  \ntn{
            \abs{\omega_k\sts{t}-\omega_k\sts{0} }^2
            +
            \abs{c_k\sts{t}-c_k\sts{0} }^2
            },
    \end{align*}
which together with \eqref{mono1} implies that
\begin{align}
\abs{ \mathcal{J}\sts{t} - \mathcal{J}\sts{0} }
    \leq
    &
    \frac{C_{\mathrm{abs}}}{ L }\sup_{0\leq s\leq t}\norm{\eps\sts{s}}_{H^1}^2
    +
    C_{\mathrm{abs}}\e^{-\theta_0 L }
    +
    C_{\mathrm{abs}}\norm{\eps\sts{0}}_{H^1}^2 \notag
  \\
  &
  +
  \sum_{k=1}^{2}
    C_{\mathrm{abs}}
    \ntn{
        \abs{\omega_k\sts{t}-\omega_k\sts{0} }^2
        +
        \abs{c_k\sts{t}-c_k\sts{0} }^2
        }. \label{est:J ref}
  \end{align}

\item{\textit{Refined estimates of $\abs{\omega_k\sts{t}-\omega_k\sts{0}}$ and $ \abs{c_k\sts{t}-c_k\sts{0}}$.}}
Recall that
\begin{align*}
  \phi\sts{t,x}=\varphi\sts{ \frac{x-\bar{x}^0 -\mu  t }{ \sqrt{t+a} } },  \phi_{\pm,0}\sts{t,x}=\varphi\sts{ \frac{x-\bar{x}^0 -\mu_{\pm,0} t }{ \sqrt{t+a} } },
    \phi_{0,\pm}\sts{t,x}=\varphi\sts{ \frac{x-\bar{x}^0 -\mu_{0,\pm} t }{ \sqrt{t+a} } },
    \end{align*}
we have
\begin{lemma}
\label{lem:massloc}
  \begin{gather}
  \label{mloc1}
    \abs{ \int\abs{u(t,x)}^2~\phi(t,x) - \int\abs{R_2\sts{t}}^2 }
    \leq
    C_{\mathrm{abs}}\e^{-\theta_0\sts{ L + \theta_0 t }} + C_{\mathrm{abs}}\norm{\eps}_{L^2}^2.
    \\
    \label{mloc2}
    \abs{ \int\abs{u(t,x)}^2~\sts{1-\phi(t,x)} - \int\abs{R_1\sts{t}}^2 }
    \leq
    C_{\mathrm{abs}}\e^{-\theta_0\sts{ L + \theta_0 t }} + C_{\mathrm{abs}}\norm{\eps}_{L^2}^2.
    \\
    \label{ploc1}
    \abs{ \Im\int\bar{u}u_x~\phi(t,x) - \Im\int\sts{\bar{R}_2\partial_xR_2}\sts{t} }
    \leq
    C_{\mathrm{abs}}\e^{-\theta_0\sts{ L + \theta_0 t }} + C_{\mathrm{abs}}\norm{\eps}_{H^1}^2.
    \\
    \label{ploc2}
    \abs{ \Im\int\bar{u}u_x~\sts{1-\phi(t,x)} - \Im\int\sts{\bar{R}_1\partial_xR_1}\sts{t} }
    \leq
    C_{\mathrm{abs}}\e^{-\theta_0\sts{ L + \theta_0 t }} + C_{\mathrm{abs}}\norm{\eps}_{H^1}^2.
   \end{gather}
\end{lemma}
\begin{proof}
By the definition of $\phi$ and the exponential decay estimate of $R_k,$ it is easy to check that
  \begin{align}
  \label{46}
  \nonumber
    &
    \abs{ \int\abs{u(t,x)}^2~\phi(t,x) - \int \abs{R_2}^2 }
    \\
  \nonumber
    =
    &
    \abs{ \int\abs{R_1+R_2+\eps}^2~\phi - \int \abs{R_2}^2 }
    \\
  \nonumber
    =
    &
    \abs{
    \int
    \sts{
     \abs{R_1}^2~\phi - \abs{R_2}^2~\sts{1-\phi} + \abs{\eps}^2~\phi
    +
    2\Re
    \left[
        \bar{R}_1R_2\phi + \bar{R}_1\eps\phi - 2\bar{R}_2\eps\sts{1-\phi}
    \right]
    }
    }
    \\
    \leq
    &
    2\int
    \ntn{
            \abs{R_1}^2~\phi
        +   \abs{R_2}^2~\sts{1-\phi}
        +   \abs{ \bar{R}_1R_2}
        +   \abs{\eps}^2
    }.
  \end{align}
First, by inserting the estimate \eqref{39} into \eqref{38}, from the definition of \(\theta_0\) in \eqref{thta6}, we have
\begin{gather}\label{40}
  \int\abs{ \bar{R}_1R_2}
  <
  C_{\mathrm{abs}}\e^{-\theta_0\sts{ L + \theta_0 t }}.
\end{gather}
Now, by
\eqref{41} and \eqref{42}, we obtain
\begin{align*}
   \frac{\d}{\d t}\left[ \bar{x}^0+\mu t -\sqrt{t+a} - x_1\sts{t} \right]
   =
  &
  \mu - \dot{x}_1(t) - \frac{1}{2\sqrt{t+a}}
  \\
  =
  &
  \mu - \sts{ \dot{x}_1(t)-c_1(t) } - \sts{ c_1(t)-c_1^0 } -c_1^0 - \frac{1}{2\sqrt{t+a}}
  \\
  \geq
  &
  \mu -c_1^0 - C_{\mathrm{abs}}\norm{\eps\sts{t}}_{H^1}
    - C_{\mathrm{abs}}\e^{-\theta_2\sts{ L+\theta_2 t } } - C_{\ift}\alpha
  \\
  \geq
  &
  \sts{\mu -c_1^0} - C_{\mathrm{abs}}C_{\ift}\alpha
    - C_{\mathrm{abs}}\e^{-\theta_2 L } - C_{\ift}\alpha
  \\
  \geq
  &
  \frac{1}{2}\sts{ \mu -c_1^0 },
\end{align*}Integrating from $0$ to $t$, we have
\begin{align}
\label{43}
  \bar{x}^0+\mu t -\sqrt{t+a} - x_1\sts{t}
  \geq
  \frac{L}{4} + \frac{1}{2}\sts{ \mu -c_1^0 } t.
\end{align}
In the similar way, we have
\begin{align}
\label{44}
  x_2\sts{t} - \sts{ \bar{x}^0+\mu t + \sqrt{t+a} }
  \geq
   \frac{L}{4} + \frac{1}{2}\sts{ c_2^0 - \mu } t.
\end{align}
By \eqref{43}, \eqref{44}, the definition of $\phi$ and  the explicit expression of $R_1$ and $R_2$, we obtain
\begin{gather}\label{45}
  \int\left[ \abs{R_1}^2~\phi + \abs{R_2}^2~\sts{1-\phi} \right]<C_{\mathrm{abs}} \e^{-\theta_0\sts{ L + \theta_0 t }}.
\end{gather}
Inserting \eqref{40} and \eqref{45} into \eqref{46}, it is easy to check that \eqref{mloc1} holds. The estimates
\eqref{mloc2}-\eqref{ploc2}
can be proved in the similar way.
\end{proof}

\begin{lemma}
 \label{lem:errnorm}
   \begin{align}
   \label{err:l21}
    &
    \int\abs{u(t,x)}^2~\sts{ \phi - \phi_{0,-} }(t,x)
    +
    \abs{ \Im\int\bar{u}u_x\sts{ \phi - \phi_{0,-} } }
    \leq
    C_{\mathrm{abs}}\e^{-\theta_0\sts{ L + \theta_0 t } }
    +
    C_{\mathrm{abs}}\norm{\eps}_{H^1}^2,
\\
\label{err:l22}
  &
  \int\abs{u(t,x)}^2~\sts{ \phi_{-,0} - \phi }(t,x)
  +
  \abs{ \Im\int\bar{u}u_x\sts{ \phi_{-,0} - \phi } }
    \leq
    C_{\mathrm{abs}}\e^{-\theta_0\sts{ L + \theta_0 t } }
    +
    C_{\mathrm{abs}}\norm{\eps}_{H^1}^2,
    \\
    \label{err:h11}
    &
    \int\abs{u(t,x)}^2~\sts{ \phi - \phi_{+,0} }(t,x)
    +
    \abs{ \Im\int\bar{u}u_x\sts{ \phi - \phi_{+,0} } }
    \leq
    C_{\mathrm{abs}}\e^{-\theta_0\sts{ L + \theta_0 t } }
    +
    C_{\mathrm{abs}}\norm{\eps}_{H^1}^2,
    \\
    \label{err:h12}
    &
    \int\abs{u(t,x)}^2~\sts{ \phi_{0,+} - \phi }(t,x)
    +
    \abs{ \Im\int\bar{u}u_x\sts{ \phi_{0,+} - \phi }}
    \leq
    C_{\mathrm{abs}}\e^{-\theta_0\sts{ L + \theta_0 t } }
    +
C_{\mathrm{abs}}\norm{\eps}_{H^1}^2.
\end{align}
\end{lemma}
\begin{proof}We only give the proof of \eqref{err:l21}. The estimates \eqref{err:l22}-\eqref{err:h12} can be shown in the similar way. Now by the definition of $\phi$ and $\phi_{0,-}$, it is easy to check that for any time $t>0$
\begin{align*}
  \supp{\sts{ \phi - \phi_{0,-} }(t,\cdot)}
    \subset
  \sts{~\bar{x}^0 + \frac{\mu}{2}t - \sqrt{t+a}~,~\bar{x}^0 + \mu t + \sqrt{t+a}~}.
\end{align*}
Then, it follows from \eqref{19} that,
\begin{align*}
  \int\abs{u(t,x)}^2~\sts{ \phi - \phi_{0,-} }(t,x)
  \leq &
  C_{\mathrm{abs}}\int_{ \supp{\sts{ \phi - \phi_{0,-} }} } \left[
    \abs{R_1}^2 + \abs{R_2}^2 + \abs{\eps}^2
  \right].
\end{align*}
Now, by
\eqref{41} and \eqref{42}, we obtain
\begin{align*}
   \frac{\d}{\d t}\left[ \bar{x}^0+\frac{\mu}{2}t -\sqrt{t+a} - x_1\sts{t} \right]
   =
  &
  \frac{\mu}{2} - \dot{x}_1(t) - \frac{1}{2\sqrt{t+a}}
  \\
  =
  &
  \frac{\mu}{2} - \sts{ \dot{x}_1(t)-c_1(t) } - \sts{ c_1(t)-c_1^0 } -c_1^0 - \frac{1}{2\sqrt{t+a}}
  \\
  \geq
  &
  \frac{\mu}{2} -c_1^0 - C_{\mathrm{abs}}\norm{\eps\sts{t}}_{H^1}
    - C_{\mathrm{abs}}\e^{-\theta_2\sts{ L+\theta_2 t } } - C_{\ift}\alpha
  \\
  \geq
  &
  \sts{\frac{\mu}{2} -c_1^0} - C_{\mathrm{abs}}C_{\ift}\alpha
    - C_{\mathrm{abs}}\e^{-\theta_2 L } - C_{\ift}\alpha
  \\
  \geq
  &
  \frac{1}{4}\sts{ \mu -2c_1^0 }.
\end{align*}
Integrating from $0$ to $t$, we have
\begin{align}
\label{47}
  \bar{x}^0+\frac{\mu}{2} t -\sqrt{t+a} - x_1\sts{t}
  \geq
   \frac{L}{4} + \frac{1}{4}\sts{ \mu -c_1^0 } t.
\end{align}
A similar argument implies that
\begin{align}
\label{48}
  x_2\sts{t} -
  \left[ \bar{x}^0+ \mu t + \sqrt{t+a} \right]
  \geq
   \frac{L}{4} + \frac{1}{4}\sts{ \mu -2c_1^0 } t.
\end{align}
Hence, we have
\begin{align*}
&   \int\abs{u(t,x)}^2~\sts{ \phi - \phi_{0,-} }(t,x)\\
  \leq &
  C_{\mathrm{abs}}\int_{ \supp{\sts{ \phi - \phi_{0,-} }} } \left[
    \abs{R_1}^2 + \abs{R_2}^2 + \abs{\eps}^2
  \right].
  \\
  \leq   &
  C_{\mathrm{abs}}\e^{-2\theta_2\sts{ \bar{x}^0 + \frac{\mu}{2}t - \sqrt{t+a} - x_1 (t)}}
  +
  C_{\mathrm{abs}}\e^{ -2\theta_2\sts{ x_2(t) - \bar{x}^0 - \mu t - \sqrt{t+a} } }
  +
  C_{\mathrm{abs}}\norm{\eps}_{L^2}^2
  \\
  \leq  &
  C_{\mathrm{abs}}\e^{- \theta_0 \sts{ L + \theta_0 t }}
  +
  C_{\mathrm{abs}}\norm{\eps}_{L^2}^2.
\end{align*}
By a similar argument as above, it is not hard to see that
\begin{align*}
  \abs{ \Im\int\bar{u}u_x\sts{ \phi - \phi_{0,-} } }
  \leq  &
  C_{\mathrm{abs}}\e^{- \theta_0 \sts{ L + \theta_0 t }}
  +
  C_{\mathrm{abs}}\norm{\eps}_{H^1}^2.
\end{align*}
This gives \eqref{err:l21}.
\end{proof}

By Lemma \ref{lem:massloc} and \ref{lem:errnorm}, we are able to show the following result.
\begin{lemma}\label{lem:mass}
  \begin{align*}
  \sum_{k=1}^{2}  \abs{M(R_k(t)) - M(R_k(0))}
  + & \sum_{k=1}^{2}
  \abs{P(R_k(t)) - P(R_k(0))}
  \\
  \leq
       &
         \frac{C_{\mathrm{abs}}}{ L }\sup_{0<s<t}\norm{\eps\sts{s}}_{H^1}^2
         +
         C_{\mathrm{abs}}\norm{\eps\sts{0}}_{H^1}^2
         +
         C_{\mathrm{abs}}\e^{-\theta_0 L }
    \\
       &
       +
      C_{\mathrm{abs}} \sum_{k=1}^{2}
         \ntn{
             \abs{\omega_k\sts{t}-\omega_k\sts{0} }^2
             +
             \abs{c_k\sts{t}-c_k\sts{0} }^2
             }.
  \end{align*}
\end{lemma}
    \begin{proof}
On one hand, from the expression of $\mathcal{J}_{+,0}\sts{t}$ and $\mathcal{J}\sts{t},$ we have
         \begin{align*}
                    \mathcal{J}_{+,0}\sts{t} - \mathcal{J}\sts{t}
          - &
          \frac{\omega_2\sts{0}-\omega_1\sts{0}}{2}  \int\abs{R_2\sts{t}}^2 \\
         =
          &
          \frac{\omega_2\sts{0}-\omega_1\sts{0}}{2}\sts{ \int\abs{u}^2~\phi - \int\abs{R_2\sts{t}}^2 }
          \\
          &
          +
          \sts{\omega_2\sts{0}-\omega_1\sts{0}}
          \int\abs{u}^2~\sts{ \phi_{+,0} - \phi} \\
          &
          -
          \frac{c_2\sts{0}-c_1\sts{0}}{2}\Im\int\bar{u}u_x\sts{ \phi_{+,0} - \phi }
          .
        \end{align*}
    Combining \eqref{mloc1} with \eqref{err:h11}, we have for any $t\in [0, T^*]$
        \begin{align*}
          \abs{ \mathcal{J}_{+,0}\sts{t} - \mathcal{J}\sts{t} - \frac{\omega_2\sts{0}-\omega_1\sts{0}}{2} \int\abs{R_2\sts{t}}^2 }
          \leq
          C_{\mathrm{abs}}\norm{\eps\sts{t}}_{H^1}^2 + C_{\mathrm{abs}}\e^{-\theta_0\sts{L + \theta_0 t }}.
        \end{align*}
Thus,
        \begin{align*}
          -\ntn{ \mathcal{J}_{+,0}\sts{t} - \mathcal{J}\sts{t} - \frac{\omega_2\sts{0}-\omega_1\sts{0}}{2} \int\abs{R_2\sts{t}}^2 }
          &
          \leq
          C_{\mathrm{abs}}\norm{\eps\sts{t}}_{H^1}^2 + C_{\mathrm{abs}}\e^{-\theta_0\sts{L + \theta_0 t }},
          \\
          \mathcal{J}_{+,0}\sts{0} - \mathcal{J}\sts{0} - \frac{\omega_2\sts{0}-\omega_1\sts{0}}{2} \int\abs{R_2\sts{0}}^2
          &
          \leq
          C_{\mathrm{abs}}\norm{\eps\sts{0}}_{H^1}^2 + C_{\mathrm{abs}}\e^{-\theta_0 L},
        \end{align*}
        which implies that
        \begin{align*}
        \nonumber
         &
            \frac{\omega_2\sts{0}-\omega_1\sts{0}}{2} \left[ \int\abs{R_2\sts{t}}^2 - \int\abs{R_2\sts{0}}^2 \right]
          +
            \sts{ \mathcal{J}\sts{t} - \mathcal{J}\sts{0} }
          -
            \sts{ \mathcal{J}_{+,0}\sts{t} -\mathcal{J}_{+,0}\sts{0} }
          \\
          \leq
          &
          C_{\mathrm{abs}}\norm{\eps\sts{0}}_{H^1}^2
          +
          C_{\mathrm{abs}}\norm{\eps\sts{t}}_{H^1}^2
          +
          C_{\mathrm{abs}}\e^{-\theta_0 L},
        \end{align*}
   which together with \eqref{est:rem} and \eqref{est:J ref} implies that
        \begin{align}
        \nonumber
          &
          \frac{\omega_2\sts{0}-\omega_1\sts{0}}{2}
          \left[
            \int\abs{R_2\sts{t}}^2 - \int\abs{R_2\sts{0}}^2
          \right]
          \\
          \nonumber
          \leq
          &
          \abs{ \mathcal{J}\sts{t} - \mathcal{J}\sts{0} }
            +
          \left[ \mathcal{J}_{+,0}\sts{t} -\mathcal{J}_{+,0}\sts{0} \right]
          +
          C_{\mathrm{abs}}\norm{\eps\sts{0}}_{H^1}^2
          +
          C_{\mathrm{abs}}\norm{\eps\sts{t}}_{H^1}^2
          +
          C_{\mathrm{abs}}\e^{-\theta_0 L}
          \\
          \nonumber
          \leq
          &
            \frac{C_{\mathrm{abs}}}{ L }\sup_{0<s<t}\norm{\eps\sts{s}}_{H^1}^2
            +
            C_{\mathrm{abs}}\norm{\eps\sts{0}}_{H^1}^2
            +
            C_{\mathrm{abs}}\norm{\eps\sts{t}}_{H^1}^2
            +
            C_{\mathrm{abs}}\e^{-\theta_0 L }
          \\
          \nonumber
          &
          +    C_{\mathrm{abs}}
          \sum_{k=1}^{2}
            \left[
                \abs{\omega_k\sts{t}-\omega_k\sts{0} }^2
                +
                \abs{c_k\sts{t}-c_k\sts{0} }^2
            \right]
          \\
          \nonumber
          \leq
          &
            \frac{C_{\mathrm{abs}}}{ L }\sup_{0<s<t}\norm{\eps\sts{s}}_{H^1}^2
            +
            C_{\mathrm{abs}}\norm{\eps\sts{0}}_{H^1}^2
            +
            C_{\mathrm{abs}}\e^{-\theta_0 L }
          \\
          \label{49}
          &
          +  C_{\mathrm{abs}}
          \sum_{k=1}^{2}
            \left[
                \abs{\omega_k\sts{t}-\omega_k\sts{0} }^2
                +
                \abs{c_k\sts{t}-c_k\sts{0} }^2
            \right].
        \end{align}
On the other hand, we have
        \begin{align*}
                    \mathcal{J}\sts{t} - \mathcal{J}_{-,0}\sts{t} - &  \frac{\omega_2\sts{0}-\omega_1\sts{0}}{4}
            \int\abs{R_2\sts{t}}^2
          \\
         =
          &
          \frac{\omega_2\sts{0}-\omega_1\sts{0}}{4}\sts{\int\abs{u}^2~\phi - \int\abs{R_2\sts{t}}^2 } \\
          &
          +
          \frac{\omega_2\sts{0}-\omega_1\sts{0}}{4}\int\abs{u}^2~\sts{ \phi - \phi_{-,0} }
          \\
          &
          -
          \frac{c_2\sts{0}-c_1\sts{0}}{2}\Im\int\bar{u}u_x~\sts{ \phi - \phi_{-,0} }.
        \end{align*}
Thus, by \eqref{mloc1} and \eqref{err:l22}, we have for any $t\in [0, T^*]$
        \begin{align*}
          &
          \abs{
            \mathcal{J}\sts{t} - \mathcal{J}_{-,0}\sts{t} - \frac{\omega_2\sts{0}-\omega_1\sts{0}}{4}
            \int\abs{R_2\sts{t}}^2
          }
          \leq
          C_{\mathrm{abs}}\norm{\eps\sts{t}}_{H^1}^2 + C_{\mathrm{abs}}\e^{-\theta_0\sts{L + \theta_0 t }},
        \end{align*}
        which implies that,
        \begin{gather*}
           \mathcal{J}\sts{t} - \mathcal{J}_{-,0}\sts{t} - \frac{\omega_2\sts{0}-\omega_1\sts{0}}{4}
            \int\abs{R_2\sts{t}}^2
          \leq
          C_{\mathrm{abs}}\norm{\eps\sts{t}}_{H^1}^2 + C_{\mathrm{abs}}\e^{-\theta_0\sts{L + \theta_0 t }},
          \\
          -\ntn{ \mathcal{J}\sts{0} - \mathcal{J}_{-,0}\sts{0} - \frac{\omega_2\sts{0}-\omega_1\sts{0}}{4}
            \int\abs{R_2\sts{0}}^2 }
          \leq
          C_{\mathrm{abs}}\norm{\eps\sts{0}}_{H^1}^2 + C_{\mathrm{abs}}\e^{-\theta_0 L}.
        \end{gather*}
        Therefore,
        \begin{align*}
        &
            \frac{\omega_2\sts{0}-\omega_1\sts{0}}{4}
             \left[
                \int\abs{R_2\sts{0}}^2 - \int\abs{R_2\sts{t}}^2
             \right]
          +
            \sts{ \mathcal{J}\sts{t} - \mathcal{J}\sts{0} }
          -
            \sts{ \mathcal{J}_{-,0}\sts{t} -\mathcal{J}_{-,0}\sts{0} }
          \\
          \leq
          &
          C_{\mathrm{abs}}
            \sup_{0\leq s\leq t}\norm{\eps\sts{0}}_{H^1}^2
          + C_{\mathrm{abs}}\e^{-\theta_0 L},
        \end{align*}
   which together with \eqref{est:rem} implies that
        \begin{align}
        \nonumber
          &
          \frac{\omega_2\sts{0}-\omega_1\sts{0}}{4}
             \left[
                \int\abs{R_2\sts{0}}^2 - \int\abs{R_2\sts{t}}^2
             \right]
          \\
          \nonumber
          \leq
          &
          \abs{ \mathcal{J}\sts{t} - \mathcal{J}\sts{0} }
            +
          \ntn{ \mathcal{J}_{-,0}\sts{t} -\mathcal{J}_{-,0}\sts{0} }
          +
            C_{\mathrm{abs}}\norm{\eps\sts{0}}_{H^1}^2
          + C_{\mathrm{abs}}\norm{\eps\sts{t}}_{H^1}^2
          + C_{\mathrm{abs}}\e^{-\theta_0 L}
          \\
          \nonumber
          \leq
          &
            \frac{C_{\mathrm{abs}}}{ L }\sup_{0<s<t}\norm{\eps\sts{s}}_{H^1}
            +
            C_{\mathrm{abs}}\norm{\eps\sts{0}}_{H^1}^2
            +
            C_{\mathrm{abs}}\e^{-\theta_0 L }
          \\
          \label{50}
          &
          +   C_{\mathrm{abs}}
          \sum_{k=1}^{2}
                     \left[
                \abs{\omega_k\sts{t}-\omega_k\sts{0} }^2
                +
                \abs{c_k\sts{t}-c_k\sts{0} }^2
            \right].
        \end{align}
Combining \eqref{49} with \eqref{50}, we have
\begin{align}
\label{56}
\nonumber
    \abs{ \int\abs{R_2\sts{t}}^2 - \int\abs{R_2\sts{0}}^2 }
    \leq
       &
         \frac{C_{\mathrm{abs}}}{ L }\sup_{0<s<t}\norm{\eps\sts{s}}_{H^1}^2
         +
         C_{\mathrm{abs}}\norm{\eps\sts{0}}_{H^1}^2
         +
         C_{\mathrm{abs}}\e^{-\theta_0 L }
    \\
       &
       +    C_{\mathrm{abs}}
       \sum_{k=1}^{2}
         \ntn{
             \abs{\omega_k\sts{t}-\omega_k\sts{0} }^2
             +
             \abs{c_k\sts{t}-c_k\sts{0} }^2
             }.
\end{align}
Similar argument implies that
\begin{align*}
    \abs{
        \Im\int\sts{\bar{R}_2\partial_xR_2}\sts{t}
        -
        \Im\int\sts{\bar{R}_2\partial_xR_2}\sts{0} }
    \leq
       &
         \frac{C_{\mathrm{abs}}}{ L }\sup_{0<s<t}\norm{\eps\sts{s}}_{H^1}^2
         +
         C_{\mathrm{abs}}\norm{\eps\sts{0}}_{H^1}^2
         +
         C_{\mathrm{abs}}\e^{-\theta_0 L }
    \\
       &
       +   C_{\mathrm{abs}}
       \sum_{k=1}^{2}
         \ntn{
             \abs{\omega_k\sts{t}-\omega_k\sts{0} }^2
             +
             \abs{c_k\sts{t}-c_k\sts{0} }^2
             }.
\end{align*}

This concludes the estimates of the solitary wave $R_2.$ In order to obtain the estimates of the solitary wave $R_1,$ we will make use of the conservation laws of mass and momentum and the orthogonality condition \eqref{53}.
Firstly, by the mass conservation and the orthogonality condition \eqref{53},
\begin{align*}
  &
  \nonumber
  \abs{
    \int\abs{R_1\sts{t}+R_2\sts{t}+\eps\sts{t}}^2
    -
    \int\abs{R_1\sts{0}+R_2\sts{0}+\eps\sts{0}}^2
  }
  \\
  \nonumber
  =
  &
  \abs{
    \sum_{k=1}^{2}
    \sts{\int\abs{R_k\sts{t}}^2-\int\abs{R_k\sts{0}}^2}
    +
    2\Re\int
    \ntn{
        \sts{\bar{R}_1R_2}\sts{t} - \sts{\bar{R}_1R_2}\sts{0}
    }
    +
    \int\sts{ \abs{\eps\sts{t}}^2 - \abs{\eps\sts{0}}^2 }
  }
  \\
  \geq
  &
  \abs{
    \sum_{k=1}^{2}
    \sts{\int\abs{R_k\sts{t}}^2-\int\abs{R_k\sts{0}}^2}
    }
    -
    2\int
    \ntn{
        \abs{\sts{\bar{R}_1R_2}\sts{t}}
        +
        \abs{\sts{\bar{R}_1R_2}\sts{0}}
    }
    -
    \int\sts{ \abs{\eps\sts{t}}^2 + \abs{\eps\sts{0}}^2 },
\end{align*}
which together with \eqref{40} implies that
\begin{align}\label{55}
  \abs{
    \sum_{k=1}^{2}
    \sts{\int\abs{R_k\sts{t}}^2-\int\abs{R_k\sts{0}}^2}
  }
  \leq
  C_{\mathrm{abs}}\e^{-\theta_0 L}
  +
  C_{\mathrm{abs}}\sts{\norm{\eps\sts{0}}_{L^2}^2 + \norm{\eps\sts{t}}_{L^2}^2}.
\end{align}
Secondly, by \eqref{est:rem} and \eqref{55}, we have
\begin{align}\label{56}
\nonumber
  \abs{
    \int\abs{R_1\sts{t}}^2-\int\abs{R_1\sts{0}}^2
  }
  \leq
  &
  \abs{
    \int\abs{R_2\sts{t}}^2-\int\abs{R_2\sts{0}}^2
  }
  +
  C_{\mathrm{abs}}\e^{-\theta_0 L}
  +
  C_{\mathrm{abs}}\sup_{0\leq s\leq t}\norm{\eps\sts{s}}_{L^2}^2
  \\
  \nonumber
  \leq
  &
  \frac{C_{\mathrm{abs}}}{ L }\sup_{0<s<t}\norm{\eps\sts{s}}_{H^1}^2
         +
         C_{\mathrm{abs}}\norm{\eps\sts{0}}_{H^1}^2
         +
         C_{\mathrm{abs}}\e^{-\theta_0 L }
    \\
       &
       + C_{\mathrm{abs}}
       \sum_{k=1}^{2}
         \ntn{
             \abs{\omega_k\sts{t}-\omega_k\sts{0} }^2
             +
             \abs{c_k\sts{t}-c_k\sts{0} }^2
             }
\end{align}
In a similar way, we have
\begin{align}\label{57}
  \nonumber
  &
  \abs{
    \Im\int\bar{R}_1\sts{t}\partial_xR_1\sts{t}
    -
    \Im\int\bar{R}_1\sts{0}\partial_xR_1\sts{0}
  }
  \\
  \nonumber
  \leq
  &
  \frac{C_{\mathrm{abs}}}{ L }\sup_{0<s<t}\norm{\eps\sts{s}}_{H^1}^2
         +
         C_{\mathrm{abs}}\norm{\eps\sts{0}}_{H^1}^2
         +
         C_{\mathrm{abs}}\e^{-\theta_0 L }
    \\
       &
       +   C_{\mathrm{abs}}
       \sum_{k=1}^{2}
         \ntn{
             \abs{\omega_k\sts{t}-\omega_k\sts{0} }^2
             +
             \abs{c_k\sts{t}-c_k\sts{0} }^2
             }.
\end{align}
This ends the proof.
\end{proof}

Next, by the nondegenerate condition $\det[ d''\sts{\omega_k^0,c_k^0} ]<0,$ for $k=1,2$,  we have for sufficiently small $\alpha$ and sufficiently large $L,$
\begin{align}\label{58}
  \det\left[ d''\sts{\omega_k\sts{0},c_k\sts{0}}\right]
  <
  \frac{1}{2}
  \det\left[ d''\sts{\omega_k^0,c_k^0} \right]<0.
\end{align}
By the smallness of $\abs{\omega_k\sts{t}-\omega_k\sts{0}}$ and $\abs{c_k\sts{t}-c_k\sts{0}}$, we have the following expression,
    \begin{align*}
        \begin{pmatrix}
            M\left(R_k(t)\right)- M\left(R_k(0)\right)  \\
            P\left(R_k(t)\right) - P\left(R_k(0)\right)
        \end{pmatrix}
        =
        &
        d''\sts{\omega_k\sts{0}, c_k\sts{ 0 }}
        \begin{pmatrix}
          \omega_k\sts{t}-\omega_k\sts{0} \\
          c_k\sts{t}-c_k\sts{0}
        \end{pmatrix}
        \\
        &
        +
        \bigo{ \abs{ \omega_k\sts{t}-\omega_k\sts{0} }^2 + \abs{ c_k\sts{t}-c_k\sts{0} }^2 },
    \end{align*}
    it follows that, for $k=1,2,$
    \begin{align*}
      &
      \abs{ \omega_k\sts{t}-\omega_k\sts{0} }
      +
      \abs{ c_k\sts{t}-c_k\sts{0} }
      \\
      <
      &
      C_{\mathrm{abs}}
        \abs{ \int\abs{R_k\sts{t}}^2-\int\abs{R_k\sts{0}}^2 }
        +
      C_{\mathrm{abs}}
        \abs{
            \Im\int\sts{\bar{R}_k\partial_xR_k}\sts{t}
            -
            \Im\int\sts{\bar{R}_k\partial_xR_k}\sts{0}
        }.
    \end{align*}
By Lemma \ref{lem:mass}, we have
\begin{align*}
  \abs{ \omega_k\sts{t}-\omega_k\sts{0} }
  +
  \abs{ c_k\sts{t}-c_k\sts{0} }
  <
         &
         \frac{C_{\mathrm{abs}}}{ L }\sup_{0\leq s\leq t}\norm{\eps\sts{s}}_{H^1}^2
         +
         C_{\mathrm{abs}}\norm{\eps\sts{0}}_{H^1}^2
         +
         C_{\mathrm{abs}}\e^{-\theta_0 L }
    \\
       &
       +
         C_{\mathrm{abs}}
        \sum_{k=1}^{2} \ntn{
             \abs{\omega_k\sts{t}-\omega_k\sts{0} }^2
             +
             \abs{c_k\sts{t}-c_k\sts{0} }^2
                      },
\end{align*}
which together with the smallness of  $\abs{\omega_k\sts{t}-\omega_k\sts{0}}$ and $\abs{c_k\sts{t}-c_k\sts{0}}$ implies
\begin{align}\label{est:para}
  \abs{ \omega_k\sts{t}-\omega_k\sts{0} }
  +
  \abs{ c_k\sts{t}-c_k\sts{0} }
  <
         &
         \frac{C_{\mathrm{abs}}}{ L }\sup_{0\leq s\leq t}\norm{\eps\sts{s}}_{H^1}^2
         +
         C_{\mathrm{abs}}\norm{\eps\sts{0}}_{H^1}^2
         +
         C_{\mathrm{abs}}\e^{-\theta_0 L }.
\end{align}

\item{\textit{ Conclusion. }} By \eqref{est:rem} and \eqref{est:para}, we obtain
\begin{align*}
 \norm{\eps\sts{t}}_{H^1}^2
    \leq
    &
    \frac{C_{\mathrm{abs}}}{ L }\sup_{0\leq s\leq t}\norm{\eps\sts{s}}_{H^1}^2
    +
    C_{\mathrm{abs}}\norm{\eps\sts{0}}_{H^1}^2
    +
    C_{\mathrm{abs}}\e^{-\theta_0 L }
  \\
  &
  +
  \frac{ C_{\mathrm{abs}} }{L^2} \sup_{0\leq s\leq t}\norm{\eps\sts{s}}_{H^1}^4
  +
  C_{\mathrm{abs}}~\e^{-2\theta_0 L } + \norm{\eps\sts{0}}_{H^1}^2\smallo{\norm{\eps\sts{0}}_{H^1}}
  \\
    \leq
    &
    \frac{C_{\mathrm{abs}}}{ L }\sup_{0\leq s\leq t}\norm{\eps\sts{s}}_{H^1}^2
    +
    C_{\mathrm{abs}}\norm{\eps\sts{0}}_{H^1}^2
    +
    C_{\mathrm{abs}}\e^{-\theta_0 L }.
\end{align*}
Taking $L$ is large enough, we have
\begin{align}
\label{22}
    \sup_{0\leq s\leq t}\norm{\eps\sts{s}}_{H^1}^2
    \leq
    C_{\mathrm{abs}}\norm{\eps\sts{0}}_{H^1}^2
    +
    C_{\mathrm{abs}}\e^{-\theta_0 L },
\end{align}
which together with \eqref{est:para} implies that
\begin{align}
\label{24}
    \sup_{0\leq s\leq t}\norm{\eps\sts{s}}^2_{H^1} + \sum^2_{k=1}
   \left[ \abs{ \omega_k\sts{t}-\omega_k\sts{0} }
  +
  \abs{ c_k\sts{t}-c_k\sts{0} } \right]
    \leq
    C_{\mathrm{abs}}\norm{\eps\sts{0}}^2_{H^1}
    +
    C_{\mathrm{abs}}\e^{-\theta_0  L }.
\end{align}
Therefore, we have
\begin{align}
\label{25}
\notag
 & \inf_{\substack{ x_2^0-x_1^0>\frac{L}{2},\\ \gamma_1^0,~\gamma_2^0\in\R }}
    \normhone{ u\sts{t,\cdot}-\sum^2_{k=1}Q_{\omega_k^{0},c_k^{0}}\sts{ \cdot-x_k^0}\e^{\i\gamma_k^0} }
  \\
  \notag
  \leqslant &
  \normhone{ u\sts{t,\cdot}-\sum^2_{k=1}Q_{\omega_k^{0},c_k^{0}}\sts{ \cdot-x_k\sts{t} }\e^{\i\gamma_k\sts{t}} }
  \\
  \notag
  \leqslant
  &
    \normhone{ u\sts{t, \cdot}-\sum^2_{k=1}Q_{\omega_k\sts{t},c_k\sts{t}}\sts{ \cdot-x_k\sts{t} }\e^{\i\gamma_k\sts{t}} }
    +
    C_{\mathrm{abs}} \sum^2_{k=1}\left(\abs{\omega_k(t)-\omega_k^0}+\abs{c_k(t)-c_k^0}\right)
  \\
  \notag
  \leqslant
  &
    \normhone{\varepsilon\sts{t}}
    +
    C_{\mathrm{abs}} \sum^2_{k=1}\left(\abs{\omega_k\sts{t}-\omega_k\sts{0}} + \abs{c_k\sts{t}-c_k\sts{0}}+\abs{\omega_k\sts{0}-\omega_k^{0}}+ \abs{c_k\sts{0}-c_k^{0}}\right)
    \\
    \notag
 \leqslant   &\;
  C_{\mathrm{abs}}\sts{\normhone{\varepsilon\sts{0}}
    +\sum^2_{k=1}\left(\abs{\omega_k\sts{0}-\omega_k^{0}}+ \abs{c_k\sts{0}-c_k^{0}}\right)} + C_{\mathrm{abs}} e^{-\theta_0 \frac{L}{2}}
   \\
 \leqslant   & \;   C_{\mathrm{abs}}C_{\ift} \left(\delta + e^{-\theta_0 \frac{L}{2}}\right).
\end{align}
By choosing $A_0>2C_{\mathrm{abs}}C_{\ift}$, we obtain a contradiction with the definition of $T^{\ast}.$ Thus, $T^{\ast}=\infty.$ This concludes the proof.
\end{enumerate}

\section*{Appendix A}\label{app:cal}

By the explicit expression of the solitary wave, we have
\begin{align*}
  \partial_xQ_{\omega,c}
  =
  \ltl{
    \partial_x\Phi_{\omega,c}
    +
    \i \frac{c}{2}\Phi_{\omega,c}
    -
    \i \frac{1}{2\sigma+2}\Phi_{\omega,c}^{2\sigma+1}
  }
  \exp \i\left\{
    \frac{c}{2}x-\frac{1}{2\sigma +
      2}\int^{x}_{-\infty}\Phi_{\omega,c}^{2\sigma}(y)\d y\right\}.
\end{align*}
Hence we have
\begin{align}\label{mass:sw}
  \int \abs{Q_{\omega,c}}^2 = \int \Phi_{\omega,c}^2,
\end{align}
and
\begin{align}
  \int \abs{\partial_xQ_{\omega,c}}^2
  =
  &
  \int \sts{\partial_x\Phi_{\omega,c}}^2
  +
  \int
  \sts{ \frac{c}{2}\Phi_{\omega,c} - \frac{1}{2\sigma+2}\Phi_{\omega,c}^{2\sigma+1} }^2 \notag
  \\
  =
  &
  \int \sts{\partial_x\Phi_{\omega,c}}^2
  +
  \frac{c^2}{4}\int \Phi_{\omega,c}^2
  +
  \frac{1}{\sts{2\sigma+2}^2}\int \Phi_{\omega,c}^{4\sigma+2}
  -
  \frac{c}{2\sigma+2}\int \Phi_{\omega,c}^{2\sigma+2}, \label{kine:sw}
\end{align}
and
\begin{align}
  \Im\int\bar{Q}_{\omega,c}\partial_xQ_{\omega,c}
  =
  &
  \Im\int \Phi_{\omega,c}
  \ltl{
    \partial_x\Phi_{\omega,c}
    +
    \i \frac{c}{2}\Phi_{\omega,c}
    -
    \i \frac{1}{2\sigma+2}\Phi_{\omega,c}^{2\sigma+1}
  } \notag
  \\
  =
  &
  \int
  \ltl{ \frac{c}{2}\Phi_{\omega,c}^2 - \frac{1}{2\sigma+2}\Phi_{\omega,c}^{2\sigma+2} }. \label{mom:sw}
\end{align}
Therefore,
 we have
\begin{align*}
  &
  2M\sts{Q_{\omega,c}}\norm{ \partial_x Q_{\omega,c}}_{L^2}^2 - 4\left[P\sts{ {Q}_{\omega,c} } \right]^2
  \\
  =
  &
  \int \Phi_{\omega,c}^2
  \ltl{
      \int \sts{\partial_x\Phi_{\omega,c}}^2
      +
      \frac{c^2}{4}\int \Phi_{\omega,c}^2
      +
      \frac{1}{\sts{2\sigma+2}^2}\int \Phi_{\omega,c}^{4\sigma+2}
      -
      \frac{c}{2\sigma+2}\int \Phi_{\omega,c}^{2\sigma+2}
  }
  \\
  &
  -
  \sts{ \int
  \ltl{ \frac{c}{2}\Phi_{\omega,c}^2 - \frac{1}{2\sigma+2}\Phi_{\omega,c}^{2\sigma+2} } }^2
  \\
  =
  &
  \int \Phi_{\omega,c}^2\int \sts{\partial_x\Phi_{\omega,c}}^2
  +
  \frac{1}{\sts{2\sigma+2}^2}\int \Phi_{\omega,c}^2\int \Phi_{\omega,c}^{4\sigma+2}
  -
  \frac{1}{\sts{2\sigma+2}^2}\sts{ \int \Phi_{\omega,c}^{2\sigma+2} }^2
  \\
  \geq
  &
  \int \Phi_{\omega,c}^2\int \sts{\partial_x\Phi_{\omega,c}}^2>0.
\end{align*}
where we use the explicit expression \eqref{phi} of $  \Phi_{\omega,c}$ in the last inequality.

\section*{Appendix B}
\label{app:expand}
In this appendix, we prove Lemma \ref{lem:tlexp}.
\begin{proof}[Proof of Lemma \ref{lem:tlexp}]First note that
\begin{align*}
  \S\sts{t}
  =
  &
  E\sts{u\sts{t}} + \J_{\mathrm{sum}}\sts{u\sts{t}}
  \\
  =
  &
  \frac{1}{2}\int\abs{ u_x\sts{t} }^2
  +
  \frac{1}{2(\sigma+1)}\Im\int_{-\infty}^\infty
  |u|^{2\sigma}\bar{u}u_x
  \\
  &
  +
    \ltl{
        \frac{ \omega_1\sts{0} }{2}\int\abs{u}^2~\phi
         -
         \frac{ c_1\sts{0} }{2}\Im\int\bar{u}u_x~\phi
    }
  \\
  &
    +
    \ltl{
        \frac{ \omega_2\sts{0} }{2}\int\abs{u}^2~\sts{1-\phi}
         -
         \frac{ c_2\sts{0} }{2}\Im\int\bar{u}u_x~\sts{1-\phi}
    }.
\end{align*}
Now we will expand the right hand side of the above equality.

\textbf{\textit{ Term $ \frac{1}{2}\int\abs{ u_x\sts{t} }^2 $ }:}
By the weak interaction \eqref{26} between the solitary waves, we have
\begin{align}
\label{t1}
\notag
  \int\abs{ u_x\sts{t} }^2
  =
  &
  \int
  \abs{
    \partial_xR_1 + \partial_xR_2 + \partial_x\eps
    }^2
  \\
  \notag
  =
  &
  \int
  \ntn{
    \abs{ \partial_xR_1 }^2 + \abs{ \partial_xR_2 }^2 + \abs{ \partial_x\eps }^2
  }
  +
  2\Re\int\partial_{x}\bar{R}_1\partial_{x}R_2
  \\
  \notag
  &
  -
  2\Re\int
  \ntn{ \partial_{x,x}R_1\eps + \partial_{x,x}R_2\eps }
  \\
  \notag
  =
  &
  \int
  \ntn{
    \abs{ \partial_xR_1 }^2 + \abs{ \partial_xR_2 }^2 + \abs{ \partial_x\eps }^2
  }
  \\
  \notag
  =
  &
  \int
  \ntn{
    \abs{ \partial_xR_1 }^2 + \abs{ \partial_xR_2 }^2 + \abs{ \partial_x\eps }^2
  }
  \\
  &
  -
  2\Re\int
  \ntn{ \partial_{x,x}R_1\eps + \partial_{x,x}R_2\eps }
  +
  \bigo{\e^{ -\theta_0 \sts{L + \theta_0 t} }}
\end{align}

\textbf{
    \textit{
        Term $\frac{\omega_2\sts{0}}{2}\int \abs{u}^2\phi$
    }:
} By \eqref{40} and the definition of $\phi,$
the simple calculations give
\begin{align}
\label{t2}
\notag
\frac{\omega_2\sts{0}}{2}\int \abs{u}^2\phi
  =
  &
  \frac{\omega_2\sts{0}}{2}\int \abs{R_1+R_2+\eps}^2\phi
  \\
  \notag
  =
  &
  \frac{\omega_2\sts{0}}{2}
  \int
  \ntn{
    \abs{R_1}^2\phi
    + \abs{R_2}^2\phi
    + \abs{\eps}^2\phi
  }
    + \omega_2\sts{0}\Re\int
    \ntn{
        R_1\bar{R}_2\phi
        +
        R_1\bar{\eps}\phi
        +
        R_2\bar{\eps}\phi
    }
  \\
  \notag
  =
  &
  \frac{\omega_2\sts{0}}{2}\Re\int
    \ntn{
        \abs{R_2}^2
        +
        2 R_2\bar{\eps}
        +
        \abs{\eps}^2\phi
    }
    +
   \frac{\omega_2\sts{0}}{2}\int
    \ntn{
    \abs{R_1}^2\phi
    +
    \abs{R_2}^2\sts{\phi - 1 }
    }
    \\
    \notag
    &
    + \omega_2\sts{0}\Re\int
    \ntn{
        R_1\bar{R}_2\phi
        +
        R_1\bar{\eps}\phi
        +
        R_2\bar{\eps}\sts{ \phi - 1 }
    }
   \\
   \notag
   =
   &
  \frac{\omega_2\sts{0}}{2}\int
    \ntn{
        \abs{R_2}^2
        +
        2 R_2\bar{\eps}
        +
        \abs{\eps}^2\phi
    }
    +
    \bigo{
        \e^{ -\theta_0\sts{ L + \theta_0 t } }
    }
    +
    \norm{\eps}_{L^2}^2\smallo{ \norm{\eps}_{L^2} }
   \\
   \notag
   =
   &
  \frac{\omega_2\sts{0}}{2}
  \int\abs{R_2}^2
  +
  \frac{\omega_2\sts{t}}{2}
  \int\abs{\eps}^2\phi
  +
  \omega_2\sts{0}\Re
  \int R_2\bar{\eps}
  \\
  \notag
  &
    +
    \bigo{
        \e^{ -\theta_0\sts{ L + \theta_0 t } }
    }
    +
    \norm{\eps}_{L^2}^2\smallo{ \norm{\eps}_{L^2} }
    +
    \bigo{ \abs{ \omega_2\sts{t}-\omega_2\sts{0} }^2 }
   \\
   \notag
   =
   &
  \frac{\omega_2\sts{0}}{2}
  \int\abs{R_2}^2
  +
  \frac{\omega_2\sts{t}}{2}
  \int\abs{\eps}^2\phi
  +
  \omega_2\sts{t}\Re
  \int R_2\bar{\eps}
  \\
  &
    +
    \bigo{
        \e^{ -\theta_0\sts{ L + \theta_0 t } }
    }
    +
    \norm{\eps}_{L^2}^2\smallo{ \norm{\eps}_{L^2} }
    +
    \bigo{ \abs{ \omega_2\sts{t}-\omega_2\sts{0} }^2 }
    ,
\end{align}
where we used the orthogonality condition $\inprod{R_2}{\eps}=0$ in the last equality,

In the similar way, we can obtain

\textbf{
    \textit{
        Term
        $ \frac{ \omega_1\sts{0} }{2}\int\abs{u}^2\sts{1-\phi}$
    }:
}
\begin{align}
\label{t3}
\notag
  \frac{\omega_1\sts{0}}{2}\int \abs{u}^2~\sts{1-\phi}
   =
   &
  \frac{\omega_1\sts{0}}{2}
  \int\abs{R_1}^2
  +
  \frac{\omega_1\sts{t}}{2}
  \int\abs{\eps}^2~\sts{1-\phi}
  +
  \omega_1\sts{t}
  \int R_1\bar{\eps}
  \\
  &
    +
    \bigo{
        \e^{ -\theta_0\sts{ L + \theta_0 t } }
    }
    +
    \norm{\eps}_{L^2}^2\smallo{ \norm{\eps}_{L^2} }
    +
    \bigo{ \abs{ \omega_1\sts{t}-\omega_1\sts{0} }^2 }
    .
\end{align}

\textbf{
    \textit{
        Term
        $ \frac{ c_2\sts{0} }{2}\Im\int\bar{u}u_x\phi$
    }:
}
\begin{align}
\label{t4}
\notag
  \frac{ c_2\sts{0} }{2}\Im\int\bar{u}u_x~\phi
   =
   &
  \frac{ c_2\sts{0} }{2}
  \Im\int\bar{R}_2\partial_xR_2
  +
  \frac{ c_2\sts{t} }{2}
  \Im\int \bar{\eps}\eps_x~\phi
  +
  c_2\sts{t}
  \int R_2\bar{\eps}
  \\
  &
    +
    \bigo{
        \e^{ -\theta_0\sts{ L + \theta_0 t } }
    }
    +
    \norm{\eps}_{H^1}^2\smallo{ \norm{\eps}_{H^1} }
    +
    \bigo{ \abs{ c_2\sts{t}-c_2\sts{0} }^2 }
    .
\end{align}

\textbf{
    \textit{
        Term
        $ \frac{ c_1\sts{0} }{2}\Im\int\bar{u}u_x\sts{1-\phi} $
    }:
}
\begin{align}
\label{t5}
\notag
  \frac{ c_1\sts{0} }{2}\Im\int\bar{u}u_x\sts{1-\phi}
   =
   &
  \frac{ c_1\sts{0} }{2}
  \Im\int\bar{R}_1\partial_xR_1
  +
  \frac{ c_1\sts{t} }{2}
  \Im\int \bar{\eps}\eps_x~\phi
  +
  c_1\sts{t}
  \int R_1\bar{\eps}
  \\
     &+
    \bigo{
        \e^{ -\theta_0\sts{ L + \theta_0 t } }
    }
    +
    \norm{\eps}_{H^1}^2\smallo{ \norm{\eps}_{H^1} }
    +
    \bigo{ \abs{ c_1\sts{t}-c_1\sts{0} }^2 }
    .
\end{align}

\textbf{\textit{Term $N\sts{u}:=\frac{1}{2(\sigma+1)}\Im\int |u|^{2\sigma}\bar{u}u_x$}:} In order to expand it, we introduce the following cut-off functions around each solitary waves,
\begin{gather*}
  g_1\sts{t,x} :=
  \begin{cases}
    1, & \mbox{if } x<x_1 + \frac{1}{16}\sts{ L + {\theta_2} t}, \\
    0, & \mbox{if } x>x_1 + \frac{1}{8}\sts{ L + {\theta_2} t},
  \end{cases}
  \\
  g_2\sts{t,x} :=
  \begin{cases}
    1, & \mbox{if } x>x_2 - \frac{1}{16}\sts{ L + {\theta_2} t}, \\
    0, & \mbox{if } x<x_2 - \frac{1}{8}\sts{ L + {\theta_2} t},
  \end{cases}
  \\
  \tilde{g} := 1-g_1-g_2£¬
\end{gather*}
where $\theta_2$ is given by Lemma \ref{lem:mod2}.
Now, we decompose $  N\sts{u}$ as following,
\begin{align*}
  N\sts{u}=N_1\sts{u} + \tilde{N}\sts{u} + N_2\sts{u},
\end{align*}
where
\begin{gather*}
  N_1\sts{u} = \frac{1}{2(\sigma+1)}\Im\int |u|^{2\sigma}\bar{u}u_x~g_1,
  \\
  N_2\sts{u} = \frac{1}{2(\sigma+1)}\Im\int |u|^{2\sigma}\bar{u}u_x~g_2,
  \\
  \tilde{N}\sts{u} = N\sts{u}-N_1\sts{u}-N_2\sts{u}.
\end{gather*}
Note that
$\abs{ R_2\sts{t,x} } <C_{\mathrm{abs}}\e^{-4\theta_0\abs{x-x_2(t)}},$ we have
\begin{align}
\label{27}
   \int \abs{R_2(t,x)~g_1(t,x) } + \int \abs{R_2(t,x)~ \sts{1-g_2(t,x)} } =\bigo{ \e^{ -\theta_0\sts{ L+\theta_0 t } } },
\end{align}
which implies that
\begin{align}
  N_1\sts{u}
  =
  &
  \frac{1}{2(\sigma+1)}\Im\int \abs{ R_1 + R_1 + \eps }^{2\sigma}
  \sts{ \bar{R}_1 + \bar{R}_2 + \bar{\eps}}
  \sts{ \partial_x{R}_1 + \partial_x{R}_2 + \eps_x }g_1 \notag
  \\
  =
  &
  N_1\sts{ R_1 }
  +
  \act{N_1'\sts{R_1}}{ R_2+\eps }
  +
  \frac{1}{2}\biact{N_1''\sts{R_1} }{\sts{R_2+\eps}}{R_2+\eps} \notag
  \\
  &
  +
  \bigo{ \int \abs{R_2}^3g_1   } + \norm{\eps}_{H^1}^2\smallo{ \norm{\eps}_{H^1} }, \label{Nu:1}
\end{align}
where
\begin{align*}
  \act{N_1'\sts{ R_1 }}{ R_2+\eps }
  =
  &
  \frac{1}{2(\sigma+1)}
  \Im\int \sigma \abs{R_1}^{2\sigma-2}\bar{R}_1^2\partial_xR_1\sts{ R_2+\eps }g_1
  \\
  &
  +
  \frac{1}{2(\sigma+1)}
  \Im\int \sts{\sigma+1} \abs{R_1}^{2\sigma}\partial_xR_1\sts{ \bar{R}_2+\bar{\eps} }g_1
  \\
  &
  +
  \frac{1}{2(\sigma+1)}
  \Im\int \abs{R_1}^{2\sigma}\bar{R}_1\sts{ \partial_xR_2+\eps_x }g_1,
\end{align*}
and
\begin{align*}
  \biact{ N_1''\sts{ R_1 } }{ \sts{ R_2+\eps } }{ R_2+\eps  }
  =
  &
  \frac{1}{2(\sigma+1)}
  \Im\int2\sigma \abs{R_1}^{2\sigma-2}\bar{R}_1^2\sts{ R_2+\eps }\sts{ \partial_xR_2 +\eps_x }g_1
  \\
  &
  +
  \frac{1}{2(\sigma+1)}
  \Im\int2\sts{\sigma+1} \abs{R_1}^{2\sigma}\sts{ \bar{R}_2+\bar{\eps} }\sts{ \partial_xR_2 +\eps_x }g_1
  \\
  &
  +
  \frac{1}{2(\sigma+1)}
  \Im\int 2\sigma\sts{\sigma+1}\abs{R_1}^{2\sigma-2}\bar{R}_1\partial_xR_1\abs{ R_2+\eps }^2 g_1
  \\
  &
  +
  \frac{1}{2(\sigma+1)}
  \Im\int \sigma\sts{\sigma-1}\abs{R_1}^{2\sigma-4}\bar{R}_1^3\partial_xR_1\sts{ R_2 +\eps }^2 g_1
  \\
  &
  +
  \frac{1}{2(\sigma+1)}
  \Im\int \sigma\sts{\sigma+1}\abs{R_1}^{2\sigma-2}R_1\partial_xR_1\sts{ \bar{R}_2 +\bar{\eps} }^2 g_1.
\end{align*}
Therefore, the decay estimates \eqref{27} implies that
\begin{gather}
\label{32}
  \act{N_1'\sts{ R_1 }}{ R_2+\eps }
  =
  \act{N_1'\sts{ R_1 }}{ \eps }
  +
  \bigo{ \e^{ -\theta_0\sts{ L+\theta_0 t } } }
  \\
  \label{33}
  \biact{ N_1''\sts{ R_1 } }{ \sts{ R_2+\eps } }{ R_2+\eps  }
  =
  \biact{ N_1''\sts{ R_1 } }{ \eps }{ \eps  }
  +
  \bigo{ \e^{ -\theta_0\sts{ L+\theta_0 t } } }
  .
\end{gather}
Now, note that
$\abs{ R_1\sts{t,x} } <C_{\mathrm{abs}}\e^{-2\theta_2\abs{x-x_1(t)}},$ we have
\begin{align}
\label{28}
   \int \abs{R_1~ g_2 } + \int \abs{R_1~ \sts{1-g_1} } =\bigo{ \e^{ -\theta_0\sts{ L+\theta_0 t } } },
\end{align}
which yields that
\begin{gather}
\label{29}
  \abs{ N\sts{ R_1 } - N_1\sts{ R_1 } }
  =
  \bigo{ \e^{ -\theta_0\sts{ L+\theta_0 t } } },
  \\
  \label{30}
  \abs{\act{N_1'\sts{ R_1 }}{ \eps }-\act{N'\sts{ R_1 }}{ \eps }}
  =
  \bigo{ \e^{ -\theta_0\sts{ L+\theta_0 t } } },
  \\
  \label{31}
  \abs{ \biact{ N_1''\sts{ R_1 } }{ \eps }{ \eps  } - \biact{ N''\sts{ R_1 } }{ \eps }{ \eps  }}
  =
  \bigo{ \e^{ -\theta_0\sts{ L+\theta_0 t } } }.
\end{gather}
Inserting \eqref{29}-\eqref{31} and \eqref{32}-\eqref{33} into \eqref{Nu:1}, we obtain that
\begin{align}
\label{36}
  N_1\sts{u}
  =
  N\sts{R_1} + \act{N'\sts{ R_1 }}{ \eps } + \frac{1}{2}\biact{ N''\sts{ R_1 } }{ \eps }{ \eps  }
  +
  \bigo{ \e^{ -\theta_0\sts{ L+\theta_0 t } } }.
\end{align}
In the similar way, we have
\begin{align}
\label{37}
  N_2\sts{u}
  =
  N\sts{R_2} + \act{N'\sts{ R_2 }}{ \eps } + \frac{1}{2}\biact{ N''\sts{ R_2 } }{ \eps }{ \eps  }
  +
  \bigo{ \e^{ -\theta_0\sts{ L+\theta_0 t } } }.
\end{align}
As for the term $\tilde{N}\sts{u}$, we have
\begin{align*}
  \tilde{N}\sts{u}
  =
  &
  \frac{1}{2(\sigma+1)}\Im\int |u|^{2\sigma}\bar{u}u_x~\sts{1-g_1-g_2} \notag
  \\
  \leq
  &
  C_{\mathrm{abs}}\norm{u}_{L^\infty}^{2\sigma}\sts{ \int \abs{u_x}^2~\sts{1-g_1-g_2}~\int \abs{u}^2~\sts{1-g_1-g_2} }^{\frac{1}{2}}.
\end{align*}
which together with the smallness of $\norm{\eps}_{H^1}$, \eqref{27} and \eqref{28} implies that
\begin{align}
\label{35}
  \tilde{N}\sts{u}
   =
   \bigo{ \e^{ -\theta_0\sts{ L+\theta_0 t } } }
   +\norm{\eps}_{H^1}^2\smallo{ \norm{\eps}_{H^1} }.
\end{align}
Combining \eqref{36} with \eqref{37} and \eqref{35}, we have
\begin{align}\label{t6}
\notag
  N\sts{u}
  =
  &
  \sum_{k=1}^{2}
  N\sts{R_k}
  +
  \sum_{k=1}^{2}\act{N'\sts{ R_k }}{ \eps }
  +
  \sum_{k=1}^{2}\frac{1}{2}\biact{ N''\sts{ R_k } }{ \eps }{ \eps  }
  \\
  &
  +
  \bigo{ \e^{ -\theta_0\sts{ L+\theta_0 t } } }
   +\norm{\eps}_{H^1}^2\smallo{ \norm{\eps}_{H^1} }.
\end{align}
Summing up \eqref{t1}, \eqref{t2}, \eqref{t3}, \eqref{t4}, \eqref{t5} and \eqref{t6}, we finish the proof.
\end{proof}

%
%

\end{document}